\documentclass[11pt]{amsart}
\usepackage{
amssymb
,amsthm
,amsmath
,mathtools
,mathdots
,mathrsfs
,url
,longtable
}
\usepackage[all]{xy}
\usepackage[top=35truemm,bottom=35truemm,left=30truemm,right=30truemm]{geometry}
\usepackage[usenames]{color}

\SelectTips{cm}{12}
\objectmargin={1pt}

\usepackage{tikz}
\usetikzlibrary{intersections,calc,arrows.meta}

\newcommand{\Z}{\mathbb{Z}}
\newcommand{\R}{\mathbb{R}}
\newcommand{\C}{\mathbb{C}}
\newcommand{\F}{\mathbb{F}}

\newcommand{\Sc}{\mathcal{S}}
\renewcommand{\AA}{\mathcal{A}}
\newcommand{\FF}{\mathcal{F}}

\newcommand{\RR}{\mathscr{R}}
\newcommand{\EE}{\mathscr{E}}

\newcommand{\GL}{\mathrm{GL}}
\newcommand{\SL}{\mathrm{SL}}
\newcommand{\SO}{\mathrm{SO}}
\newcommand{\Sp}{\mathrm{Sp}}

\newcommand{\Cusp}{\mathscr{C}}
\newcommand{\Rep}{\mathrm{Rep}}
\newcommand{\Irr}{\mathrm{Irr}}
\newcommand{\unit}{\mathrm{unit}}
\newcommand{\temp}{\mathrm{temp}}
\newcommand{\Ind}{\mathrm{Ind}}
\newcommand{\Jac}{\mathrm{Jac}}
\newcommand{\supp}{\mathrm{supp}}
\newcommand{\sgn}{\mathrm{sgn}}
\newcommand{\soc}{\mathrm{soc}}
\newcommand{\ev}{\mathrm{ev}}
\newcommand{\Res}{\mathrm{Res}}
\newcommand{\FRP}{\mathrm{FRP}}
\newcommand{\Ad}{\mathrm{Ad}}
\newcommand{\good}{\mathrm{good}}
\newcommand{\bad}{\mathrm{bad}}

\newcommand{\iif}{&\quad&\text{if }}
\newcommand{\other}{&\quad&\text{otherwise}}
\newcommand{\resp}{resp.~}
\renewcommand{\1}{\mathbf{1}}
\newcommand{\ep}{\varepsilon}
\newcommand{\mm}{\mathfrak{m}}

\newcommand{\tl}[1]{\widetilde{#1}}
\newcommand{\half}[1]{\frac{#1}{2}}
\newcommand{\pair}[1]{\langle #1 \rangle}

\theoremstyle{plain}
\newtheorem{thm}{Theorem}[section]
\newtheorem{lem}[thm]{Lemma}
\newtheorem{prop}[thm]{Proposition}
\newtheorem{cor}[thm]{Corollary}
\newtheorem{qu}[thm]{Question}

\theoremstyle{definition}
\newtheorem{defi}[thm]{Definition}
\newtheorem{rem}[thm]{Remark}

\newenvironment{talign*}
{\csname align*\endcsname}
{\endalign}

\title[
Unitary dual of $\SO_{2n+1}$ and $\Sp_{2n}$:
The good parity case (and slightly beyond)
]{
Unitary dual of $p$-adic split $\SO_{2n+1}$ and $\Sp_{2n}$:
\\ The good parity case (and slightly beyond)
}
\author{Hiraku Atobe \and Alberto M\'inguez}
\date{}
\address{
Department of Mathematics, Kyoto University, 
Kitashirakawa-Oiwake-cho, Sakyo-ku, Kyoto, 606-8502, Japan
}
\email{
atobe@math.kyoto-u.ac.jp
}
\address{
Fakult{\"a}t f{\"u}r Mathematik, University of Vienna, Oskar-Morgenstern-Platz 1, 1090 Wien, Austria
}
\email{
alberto.minguez@univie.ac.at
}

\allowdisplaybreaks

\begin{document}
\maketitle

\begin{abstract}
Let $F$ be a $p$-adic field, 
and let $G$ be either the split special orthogonal group $\SO_{2n+1}(F)$ 
or the symplectic group $\Sp_{2n}(F)$, with $n \geq 0$. 
We prove that a smooth irreducible representation of good parity of $G$ is unitary 
if and only if it is of Arthur type. 
Combined with the algorithms of \cite{At-Is_Arthur, HLL} for detecting Arthur type representations, 
our result leads to an explicit algorithm for checking 
the unitarity of any given irreducible representation of good parity. 
Finally, we determine the set of unitary representations 
that may appear as local components of the discrete automorphic spectrum.
\end{abstract}

%\section{Introduction}
%\section{Introduction}
\section{Introduction}
The classification of irreducible unitary representations of classical groups over local fields 
stands as a central problem in the representation theory of reductive groups. 
This program traces its origins to I. Gel'fand's work in the 1940s and 1950s on harmonic analysis 
and representations of locally compact groups. 
For general linear groups, 
the classification was completed by D. Vogan \cite{V} for real groups 
and by M. Tadi\'c \cite{T-0} for $p$-adic groups in the 1980s. 
For inner forms of general linear $p$-adic groups, 
the classification was achieved through the combined work of many mathematicians, 
see in particular \cite{T-1,Sec, Bad, BR, BHLS}. 
In the case of classical groups over a $p$-adic field, only partial results are known: 
namely, for generic and unramified representations \cite{LMT, MT}, 
and for groups of rank less than or equal to three \cite{T-rank3}. 
For recent progress in the case of real groups, see \cite{ALTV}. 
The unitary dual carries a natural topology, 
and one of the main challenges for classical groups is the large number of isolated representations 
(see \cite{T-Arthur}).
\vskip 10pt

Let $F$ be a non-archimedean local field. 
For $\GL_n(F)$ and its inner forms, 
Tadi\'c showed that an irreducible representation is unitary 
if and only if it is parabolically induced from certain representations 
called essentially \emph{Speh representations}.  
Note that unitary Speh representations arise as local components of discrete automorphic representations. 
A simplification of Tadi\'c's proof was achieved in \cite{LM}, 
where one of the key steps was to avoid the use of a theorem of J. Bernstein \cite{Ber} 
that played a central role in the original proof. 
Bernstein's method relies on special properties of the mirabolic subgroup of $\GL_n(F)$ 
that do not extend to other groups, 
and thus the extension to inner forms of $\GL_n(F)$ required significantly heavier machinery, 
just to mention some \cite{BK, BM, DKV, Sec2, Sec3, Sec4, SS}. 
In contrast, the approach in \cite{LM} relies only on elementary combinatorics 
and a careful analysis of Jacquet modules, 
what are now called \emph{$\rho$-derivatives} (see Section~\ref{DandS}), 
to control the irreducibility of the socle of parabolically induced representations.
\vskip 10pt

One of the most remarkable achievements in recent number theory is 
J. Arthur's classification \cite{Ar} of square-integrable automorphic representations 
of quasi-split symplectic and orthogonal groups. 
This breakthrough, accomplished through the collective effort of many mathematicians, 
relies on the twisted trace formula and an intricate inductive procedure, known as endoscopy. 
Arthur introduced what are now called \emph{$A$-parameters},
which are modifications of $L$-parameters designed 
to capture the non-tempered behavior of local components of discrete automorphic representations, 
and in particular, to account for the failure of the naive Ramanujan conjecture beyond $\GL_N$. 
In this paper, these local components are referred to as representations \emph{of Arthur type}.
They are unitary representations (\cite[Theorem 2.2.1]{Ar}) 
and they generalize the unitary Speh representations discussed above to the setting of classical groups.
\vskip 10pt

Representations of Arthur type were classified by C. M{\oe}glin \cite{Moe1, Moe2, Moe3, Moe4, Moe5} 
by purely local methods (see also \cite{X2}). 
Let $G$ be a symplectic or split special odd orthogonal group 
over a non-archimedean local field $F$ of characteristic zero. 
Given an $A$-parameter $\psi$ for $G$, one can decompose it as
\[
\psi = \psi_{\bad}^- \oplus \psi_{\good} \oplus \psi_{\bad}^+,
\]
where $\psi_{\good}$ is a sum of irreducible self-dual representations of the same type as $\psi$, 
and $\psi_{\bad}^- = (\psi_\bad^+)^\vee$ is a sum of irreducible representations of a different type 
(see Section \ref{sec.classical} for more details). 
The parameter $\psi$ is said to be of \emph{good parity} if $\psi_{\bad}^- = 0$.
Similarly, one can define this notion for irreducible representations of $G(F)$. 
\par

M{\oe}glin \cite{Moe1} showed that the representations associated to $\psi$ 
are irreducible parabolic inductions of the representations attached to $\psi_{\good}$, 
together with products of unitary Speh representations corresponding to $\psi_{\bad}^-$. 
This reduces the classification of Arthur type representations to the case of parameters of good parity. 
In \cite{Moe2,Moe3,Moe4,Moe5}, 
she ultimately classified these representations using a combination of techniques, 
including the analysis of reducibility in parabolic induction, 
the use of Jacquet functors, and what she called the \emph{partial Aubert involution}.
\vskip 10pt

M{\oe}glin's classification was later simplified by the first-named author in the paper \cite{At-const}, 
who used $\rho$-derivatives to reduce the construction of representations of Arthur type of good parity 
to those with non-negative discrete diagonal restriction (which are easy to determine). 
This naturally led to the question of whether the idea in \cite{LM} 
could be adapted to classify unitary representations of classical groups. 
In this paper, we carry out this program for unitary representations of good parity of $G(F)$, 
where $G$ is a symplectic or split special odd orthogonal group 
over a non-archimedean local field $F$ of characteristic zero. 
Our main theorem is as follows:

\begin{thm}\label{MAINTHM}
Let $\pi$ be an irreducible representation of $G(F)$ of good parity. 
Then $\pi$ is unitary if and only if it is of Arthur type.
\end{thm}

In fact, we prove a stronger result that will be useful for the eventual classification of the full unitary dual. 
Namely, we show that the good-parity part of any irreducible unitary representation is always of Arthur type.
See Theorem \ref{unitary_gp} for the details. 
\vskip 10pt

Our result proves a conjecture of Tadi\'c, 
refined in \cite[Conjecture 1.2]{HJLLZ}, 
and establishes a deep connection between unitary representations 
and local components of discrete automorphic representations. 
In these works, it was also conjectured (see \cite[Conjecture 1.1]{T-Arthur}) that 
all isolated unitary representations should be of good parity. 
We therefore view our theorem as a meaningful step 
toward the full classification of the unitary dual of classical groups. 
Moreover, since algorithms were developed in \cite{At-Is_Arthur, HLL} 
to determine whether a given irreducible representation of good parity is of Arthur type, 
our result shows that unitarity can be checked effectively in this case.
\vskip 10pt

We now outline the main steps behind the proof of Theorem \ref{MAINTHM}. 
One direction of the theorem follows from Arthur's work: 
every representation of Arthur type is known to be unitary (\cite[Theorem 2.2.1]{Ar}). 
For the converse, we begin by introducing some notation. 
Let $G_n$ denote either the split special orthogonal group $\SO_{2n+1}(F)$ 
or the symplectic group $\Sp_{2n}(F)$ of rank $n$. 
If $\pi$ (\resp $\tau_i$) is a smooth representation of $G_{n_0}$ (\resp $\GL_{d_i}(F)$), 
with $d_1 + \cdots + d_r + n_0 = n$, 
it is customary to denote by
\[
\tau_r \times \dots \times \tau_1 \rtimes \pi 
\]
the normalized parabolically induced representation from the standard parabolic subgroup $P$ of $G_n$ 
with Levi subgroup isomorphic to $\GL_{d_r}(F) \times \dots \times \GL_{d_1}(F) \times G_{n_0}$, 
induced from $\tau_r \boxtimes \dots \boxtimes \tau_1 \boxtimes \pi$.
\par

Let now $\pi$ be an irreducible unitary representation of good parity. 
A key input we use is that for any irreducible unitary representation $\sigma$ of a general linear group, 
the parabolically induced representation
\[
\sigma \rtimes \pi
\]
is also unitary and admissible, hence semisimple. 
By choosing $\sigma$ appropriately 
and using the explicit nature of the classification in \cite{At-const, At-Is_Arthur}, 
we are able to deduce that $\pi$ must be of Arthur type.
\vskip 10pt

Even though the underlying idea is conceptually simple, 
its implementation is technically involved. 
For any irreducible representation $\pi$ of $G_n$ of good parity, 
we define its \emph{SZ-decomposition} (see Section \ref{sec.SZ}): 
this consists of a natural sequence of irreducible representations $\tau_i$ of $\GL_{d_i}(F)$ 
(for $1 \leq i \leq r$),
each induced from $k_i$ copies of a segment-type representation, 
and a representation $\pi_0$ of $G_{n_0}$, with $d_1 + \cdots + d_r + n_0 = n$, 
such that $\pi$ is the socle of the parabolically induced representation
\[
\tau_r \times \dots \times \tau_1 \rtimes \pi_0.
\]
The representation $\pi_0$ has a specific constraint on its $\rho$-derivatives.
Roughly speaking, it has almost no nonzero derivatives. 
In Theorem~\ref{initial}, 
we prove that any irreducible representation $\pi_0$ with this property is of Arthur type. 
\par

Let $\pi_i$ denote the (irreducible) socle of $\tau_i \times \dots \times \tau_1 \rtimes \pi_0$. 
We then prove the following:

\begin{thm}[Theorem \ref{inductive}]\label{MAINTHM2}
Let $\pi$ be an irreducible representation of $G_n$ of good parity. 
If $\pi$ is unitary and $\pi_{i-1}$ is of Arthur type, 
then $\pi_{i}$ is also of Arthur type. 
\end{thm}

As a result, Theorem \ref{MAINTHM2} implies Theorem~\ref{MAINTHM}.
\vskip 10pt

Finally, we give some details about the proof of Theorem \ref{MAINTHM2}. 
Given $\tau_i$, we construct $S$, a product of $k_i$ copies of a unitary Speh representation, 
such that the socle of $\tau_i \times S$ is a product of $k_i$ copies of essentially Speh representations. 
We show that the unitarity of $\pi$ implies that  
\[
S \rtimes \pi_i
\quad\text{and}\quad
\soc(\tau_i \times S) \rtimes \pi_{i-1}
\]  
share a common irreducible subrepresentation. 
This imposes a strong constraint, which we exploit in the following step.
\par

Let $S'$ be a product of $k_i$ copies of ``long'' essentially Speh representations 
such that the socle of $S' \times \tau_i \times S$ is itself 
a product of $k_i$ copies of essentially Speh representations. 
We prove that the $S$-derivative (see Section \ref{sec.inequalities}) of  
\[
\soc\left( \soc(S' \times \tau_i \times S) \rtimes \pi_{i-1} \right)
\]  
contains the socle of $S' \rtimes \pi_i$ (Lemma \ref{geometric}). 
By choosing $S'$ appropriately, 
and using M{\oe}glin's construction, 
we obtain some information about the Langlands data for $\soc( S' \rtimes \pi_i)$.
Using it together with the algorithm in \cite{At-Is_Arthur}, 
we deduce that $\pi_i$ is of Arthur type. 
See Section \ref{sec.proof_inductive}.
\vskip 10pt

In the final two sections of this article, we explore what lies beyond the good parity case.
The examples of generic and unramified unitary representations suggest that the methods developed here
are insufficient to fully address the general case,
and that additional analytic tools will likely be necessary
(see Section~\ref{sec.beyond} for further discussion).
However, for global applications, 
a particularly interesting set that extends slightly beyond the good parity case
is that of representations attached to parameters which might be localizations of global $A$-parameters.
Namely, we will consider 
an irreducible representation of the form 
\[
\pi = S_1|\cdot|^{x_1} \times \dots \times S_r|\cdot|^{x_r} \rtimes \pi_0, 
\]
where $\pi_0$ is of Arthur type of good parity, $S_i$ is a unitary Speh representation, 
and $0 \leq x_i < \half{1}$ for $1 \leq i \leq r$.
It is well known that not all such representations are unitary
(see \cite[Example 5.1(1)]{HJLLZ}).
In Theorem~\ref{weak}, we determine which of them are unitary, confirming \cite[Conjecture 5.9]{HJLLZ}.
\par

Let us now say a few words about our restriction to symplectic and split odd orthogonal groups. 
We focus on these cases because the classification in \cite{At-const} is currently limited to them. 
However, with the extension of the results in \cite{At-Jac} to all quasi-split classical groups 
(see \cite[Appendix C]{AGIKMS}), 
we expect that our methods can be adapted to this broader setting. 
In particular, we conjecture that all unitary representations of good parity 
are of Arthur type for all quasi-split classical groups. 
It would also be interesting to investigate whether an analogue of our result holds for real groups.
\par

\subsection*{Acknowledgement}
We gratefully acknowledge the hospitality of Kyoto University during A. M\'inguez's visit in April 2025,
where the last part of this work was developed. 
We would like to thank Erez Lapid and Marko Tadi\'c for useful discussions. 
We also thank Max Gurevich for providing a counterexample to a statement 
we initially believed to be true (see Remark \ref{Max}). Finally, we thank Alexander Stadler for his help in the proof of Theorem \ref{weak}.
H. Atobe was partially supported by JSPS KAKENHI Grant Number 23K12946.
A. M\'inguez was partially funded by 
the Principal Investigator project PAT-4832423 of the Austrian Science Fund (FWF).

%\section{Preliminaries}
%\section{Preliminaries}
\section{Preliminaries}
Throughout this paper, $F$ denotes a non-archimedean local field of characteristic zero. 
In this section, we recall some basic results from representation theory of $p$-adic groups.

%\subsection{Derivatives and socles}
\subsection{Derivatives and socles}\label{DandS}
Fix $G \in \{\GL_n(F), \SO_{2n+1}(F), \Sp_{2n}(F)\}$, assumed to be split over $F$.
Let $\Rep(G)$ denote the category of smooth representations of $G$ of finite length. 
A representation $\Pi \in \Rep(G)$ is called \emph{unitary} 
if it admits a positive-definite $G$-invariant hermitian form. 
Note that every unitary representation is semisimple since it is of finite length and hence admissible.
\par

We write $\Irr(G)$ for the set of irreducible admissible representations of $G$, 
and denote by $\Irr_\unit(G)$ and $\Irr_\temp(G)$ 
the subsets consisting of unitary and tempered representations, respectively. 
Thus, we have the chain of inclusions:
\[
\Irr(G) \supset \Irr_\unit(G) \supset \Irr_\temp(G).
\]
\par

For $\Pi \in \Rep(G)$, its semisimplification is denoted by $[\Pi]$. 
This is viewed as an element of the Grothendieck group
\[
\RR(G) = \bigoplus_{\pi \in \Irr(G)} \Z [\pi],
\]
and we often identify $\pi \in \Irr(G)$ with its image $[\pi]$ in $\RR(G)$. For $\Pi_1, \Pi_2 \in \Rep(G)$, we write
\[
[\Pi_1] \geq [\Pi_2]
\]
if there exists $\Pi \in \Rep(G)$ such that $[\Pi] = [\Pi_1] - [\Pi_2] \in \RR(G)$.
\par

The \emph{socle} of $\Pi$, denoted by $\soc(\Pi)$, is its maximal semisimple subrepresentation. 
We say that $\Pi$ is \emph{socle irreducible} (SI) 
if $\soc(\Pi)$ is irreducible and appears in $[\Pi]$ with multiplicity one.
\par

For a standard parabolic subgroup $P$ of $G$, 
we denote the normalized Jacquet module of $\Pi$ along $P$ by $\Jac_P(\Pi)$. 
When $P$ is clear from context, we simply write $\Jac(\Pi)$. 
\par

Given $\tau \in \Rep(\GL_d(F))$ and a character $\chi$ of $F^\times$, 
we denote by $\tau \chi$ the twist $\tau \otimes \chi \circ \det$.
Let $\Cusp^\GL$ denote 
the set of equivalence classes of supercuspidal representations of $\GL_d(F)$ for some $d \geq 0$.
For $\rho \in \Cusp^\GL$, 
its \emph{exponent} is the unique real number $e(\rho)$ such that $\rho|\cdot|^{-e(\rho)}$ is unitary.
\par

Fix two positive integers $d \leq n$. 
Let $G \in \{\GL_n(F), \SO_{2n+1}(F), \Sp_{2n}(F)\}$, 
and let $P = P_d = MN$ be the standard maximal parabolic subgroup of $G$ 
whose Levi component is $M \cong \GL_d(F) \times G_0$, 
where $G_0$ is a classical group of the same type as $G$. 
For $\tau \in \Rep(\GL_d(F))$ and $\pi_0 \in \Rep(G_0)$, 
the normalized parabolic induction $\Ind_P^G(\tau \boxtimes \pi_0)$ is denoted by
\[
\tau \rtimes \pi_0.
\]
When $G = \GL_n(F)$, this is also written as $\tau \times \pi_0$.
\par

Fix  now $\rho$ a supercuspidal representation of $\GL_d(F)$. 
Given $\Pi \in \Rep(G)$, suppose that we have a decomposition
\[
[\Jac(\Pi)] = \sum_{\tau \in \Irr(\GL_d(F))} [\tau] \otimes [\Pi_\tau]
\]
in the Grothendieck group $\RR(\GL_d(F)) \otimes \RR(G_0)$. 
Then the \emph{$\rho$-derivative} of $\Pi$ is defined by
\[
D_\rho(\Pi) = [\Pi_\rho].
\]
If $n < d$, we simply set $D_\rho(\Pi) = 0$.
\par

If $x - y \in \Z$, we define
\[
D_{\rho|\cdot|^{x}, \dots, \rho|\cdot|^y}
\]
to be the composition $D_{\rho|\cdot|^y} \circ \cdots \circ D_{\rho|\cdot|^x}$,
where the exponents in the subscript are ordered with a fixed sign $\epsilon \in \{\pm1\}$, 
depending on the sign of $x - y$.
\par

Moreover, for $k \geq 0$, we define the \emph{$k$-th $\rho$-derivative} by
\[
D_{\rho}^{(k)}(\Pi) = \frac{1}{k!} \underbrace{D_\rho \circ \cdots \circ D_\rho}_{k\text{ times}}(\Pi),
\]
which is well-defined as an element of an appropriate Grothendieck group.

%\subsection{Representations of the Weil--Deligne--Arthur group}
\subsection{Representations of the Weil--Deligne--Arthur group}
For $G = \GL_n(F)$, $\SO_{2n+1}(F)$, or $\Sp_{2n}(F)$, the \emph{Langlands dual group} $\widehat{G}$ is defined by
\[
\widehat{G} = 
\begin{cases}
\GL_n(\C) & \text{if } G = \GL_n(F), \\
\Sp_{2n}(\C) & \text{if } G = \SO_{2n+1}(F), \\
\SO_{2n+1}(\C) & \text{if } G = \Sp_{2n}(F).
\end{cases}
\]
\par

Let $W_F$ denote the Weil group of $F$.
A \emph{representation} of $W_F \times \SL_2(\C) \times \SL_2(\C)$ is a homomorphism
\[
\psi \colon W_F \times \SL_2(\C) \times \SL_2(\C) \to \GL_n(\C)
\]
such that $\psi(W_F)$ consists of semisimple elements 
and the restriction $\psi|_{\SL_2(\C) \times \SL_2(\C)}$ is algebraic.
\par

Denote by $\Psi^+(G)$ the set of equivalence classes of such homomorphisms with image in $\widehat{G}$. We say that:
\begin{itemize}
\item $\psi \in \Psi^+(G)$ is an \emph{$A$-parameter} for $G$ if $\psi(W_F)$ is bounded;
\item $\psi \in \Psi^+(G)$ is an \emph{$L$-parameter} for $G$ if it is trivial on the second copy of $\SL_2(\C)$.
\end{itemize}
We write $\Psi(G)$ (\resp $\Phi(G)$) 
for the subset of $\Psi^+(G)$ consisting of $A$-parameters (\resp $L$-parameters). 
Through the embedding 
\[
W_F \times \SL_2(\C) \hookrightarrow W_F \times \SL_2(\C) \times \SL_2(\C), \quad (w, g) \mapsto (w, g, \1),
\]
we identify each $\phi \in \Phi(G)$ 
with a homomorphism $\phi \colon W_F \times \SL_2(\C) \to \widehat{G}$.
We define the set of \emph{tempered $L$-parameters} for $G$ as
\[
\Phi_\temp(G) = \Phi(G) \cap \Psi(G).
\]
In other words, $\phi \in \Phi(G)$ is tempered if and only if $\phi(W_F)$ is bounded.
\par

By the local Langlands correspondence, established by Harris--Taylor \cite{HT} and Henniart \cite{He}, 
we can view each $\rho \in \Cusp^\GL$ as an irreducible representation of $W_F$. 
It has bounded image if and only if $\rho$ is unitary, i.e., $e(\rho) = 0$.
\par

Let $S_a$ denote the unique irreducible algebraic representation of $\SL_2(\C)$ of dimension $a$. 
Then any representation $\psi$ of $W_F \times \SL_2(\C) \times \SL_2(\C)$ can be decomposed as
\[
\psi = \bigoplus_{i=1}^r \rho_i \boxtimes S_{a_i} \boxtimes S_{b_i}
\]
for some $\rho_1, \dots, \rho_r \in \Cusp^\GL$.

%\subsection{Representations of $\GL_n(F)$}
\subsection{Representations of $\GL_n(F)$}
A \emph{segment} is a set of the form 
\[
[x,y]_\rho = \{\rho|\cdot|^x, \rho|\cdot|^{x-1}, \dots, \rho|\cdot|^y\},
\]
where $\rho \in \Cusp^\GL$ and $x,y \in \R$ with $x-y \in \Z_{\geq0}$. 
One can associate to it two irreducible representations $\Delta_\rho[x,y]$ and $Z_\rho[y,x]$ 
uniquely determined by 
\[
\Delta_\rho[x,y] \hookrightarrow 
\rho|\cdot|^x \times \rho|\cdot|^{x-1} \times \dots \times \rho|\cdot|^y
\twoheadrightarrow Z_\rho[y,x].
\]
For $-A \leq B \leq A$, 
the unique irreducible subrepresentation of 
\[
\Delta_\rho[B, -A] \times \Delta_\rho[B+1, -A+1] \times \dots \times \Delta_\rho[A, -B]
\]
is denoted by 
\[
\Sp(\rho, a,b) = 
\begin{pmatrix}
B & \ldots & A \\
\vdots & \ddots & \vdots\\
-A & \ldots & -B
\end{pmatrix}_\rho, 
\]
where we set $a = A+B+1$ and $b = A-B+1$. 
We call it a \emph{Speh representation}.
It is unitary if and only if $\rho$ is unitary.
\par

By the local Langlands correspondence, there is a canonical bijection
\[
\Phi(\GL_n(F)) \xrightarrow{1:1} \Irr(\GL_n(F)), \quad 
\phi \mapsto \tau_\phi.
\]
If $\phi = \bigoplus_{i=1}^r \rho_i \boxtimes S_{a_i}$ with $e(\rho_1) \leq \dots \leq e(\rho_r)$, then $\tau_\phi$ is the unique irreducible subrepresentation of
\[
\Delta_{\rho_1}[x_1,-x_1] \times \dots \times \Delta_{\rho_r}[x_r,-x_r]
\]
with $x_i = \frac{a_i - 1}{2}$. 
In this case, we write
\[
\tau_\phi = L(\Delta_{\rho_1}[x_1,-x_1], \dots, \Delta_{\rho_r}[x_r,-x_r])
\]
or $\tau_\phi = L(\mathfrak{m})$, 
where
\[
\mathfrak{m} = [x_1,-x_1]_{\rho_1} + \dots + [x_r,-x_r]_{\rho_r}
\]
is the associated multisegment. Here, a \emph{multisegment} is a multi-set of segments, 
viewed as a finite formal sum.
\par

In particular, for $\psi = \bigoplus_{i=1}^r \rho_i \boxtimes S_{a_i} \boxtimes S_{b_i} \in \Psi(\GL_n(F))$, 
the representation $\tau_\psi = \tau_{\phi_\psi}$ associated to the $L$-parameter
\[
\phi_\psi \colon 
W_F \times \SL_2(\C) \ni (w,g) 
\mapsto \psi\left(w, g, 
\begin{pmatrix}
|w|^{1/2} & 0 \\
0 & |w|^{-1/2}
\end{pmatrix}
\right)
\]
is the irreducible parabolic induction of unitary Speh representations:
\[
\tau_\psi = \bigtimes_{i=1}^r \Sp(\rho_i, a_i, b_i).
\]

%\subsection{Representations of classical groups}
\subsection{Representations of classical groups}\label{sec.classical}
In this  subsection, 
we let $G$ be either $\SO_{2n+1}(F)$ or $\Sp_{2n}(F)$. 
We regard $\psi \in \Psi^+(G)$ as a self-dual representation of $W_F \times \SL_2(\C) \times \SL_2(\C)$
of symplectic type or of orthogonal type. 
We decompose
\[
\psi = \psi_\good \oplus \psi_\bad
\]
such that for each irreducible summand $\rho \boxtimes S_a \boxtimes S_b$ of $\psi$,
it is a summand of $\psi_\good$ if and only if 
$e(\rho) \in (1/2)\Z$ and $\rho|\cdot|^{-e(\rho)} \boxtimes S_{a+2|e(\rho)|} \boxtimes S_b$ 
is self-dual of the same type as $\psi$. 
We call $\psi$ \emph{of good parity} (\resp \emph{of bad parity}) 
if $\psi = \psi_\good$ (\resp $\psi = \psi_\bad$).
\par

For an $A$-parameter $\psi \in \Psi(G)$, 
if $\psi_\good = \oplus_{i=1}^r \rho_i \boxtimes S_{a_i} \boxtimes S_{b_i}$, 
we set 
\[
A_\psi = \bigoplus_{i=1}^r \Z/2\Z e(\rho_i, a_i, b_i). 
\]
Namely, it is a free $\Z/2\Z$-module with a canonical basis $\{e(\rho_i, a_i, b_i)\}_{i=1,\dots,r}$. 
Its quotient by the subgroup generated by 
\begin{itemize}
\item
$e(\rho_i, a_i, b_i)+e(\rho_j, a_j, b_j)$ such that 
$\rho_i \boxtimes S_{a_i} \boxtimes S_{b_i} \cong \rho_j \boxtimes S_{a_j} \boxtimes S_{b_j}$; 
and 
\item
$z_\psi = \sum_{i=1}^r e(\rho_i, a_i, b_i)$
\end{itemize}
is denoted by $\AA_\psi$. 
We often regard the Pontryagin dual $\widehat{\AA_\psi}$ of $\AA_\psi$
as a subgroup of the Pontryagin dual  $\widehat{A_\psi}$ of $A_\psi$. 
For $\ep \in \widehat{A_\psi}$, 
we set $\ep(\rho_i \boxtimes S_{a_i} \boxtimes S_{b_i}) = \ep(e(\rho_i,a_i,b_i))$. 
If $\psi = \phi$ is a tempered $L$-parameter, 
we write $e(\rho_i, a_i) = e(\rho_i, a_i, 1)$ and $\ep(\rho_i \boxtimes S_{a_i}) = \ep(e(\rho_i,a_i))$.
\par

By Arthur's endoscopic classification (\cite[Theorem 2.2.1]{Ar}), 
for $\psi \in \Psi(G)$, 
there is a multi-set $\Pi_\psi$ over $\Irr_\unit(G)$, 
called the \emph{$A$-packet} associated to $\psi$, 
together with a map
\[
\Pi_\psi \rightarrow \widehat{\AA_\psi},\; \pi \mapsto \pair{\cdot, \pi}_\psi
\]
characterized by certain endoscopic character identities. 
We call an irreducible representation $\pi$ of $G$ \emph{of Arthur type} 
if there is $\psi \in \Psi(G)$ such that $\pi \in \Pi_\psi$.
In \cite{Moe1,Moe2,Moe3,Moe4,Moe5},
M{\oe}glin explicitly constructed the $A$-packet $\Pi_\psi$, 
and in particular, she showed in \cite{Moe5} that $\Pi_\psi$ is multiplicity free, 
i.e., is a subset of $\Irr_\unit(G)$. 
For her construction, see also \cite{X2} or Section \ref{sec.Moe} below.
\par

As in \cite[Theorem 2.2.1]{Ar}, 
if $\psi = \phi$ is a tempered $L$-parameter, 
then $\Pi_\phi$ is a subset of $\Irr_\temp(G)$, 
the map $\Pi_\phi \rightarrow \widehat{\AA_\phi}$ is bijective, 
and 
\[
\Irr_\temp(G) = \bigsqcup_{\phi \in \Phi_\temp(G)} \Pi_\phi.
\]
For $\ep \in \widehat{\AA_\phi}$, 
the corresponding element in $\Pi_\phi$ is denoted by $\pi(\phi, \ep)$.
\par

By the Langlands classification, 
one can extend this classification of $\Irr_\temp(G)$ to $\Irr(G)$. 
For $\phi \in \Phi(G)$, if we write
\[
\phi = \left(\bigoplus_{i=1}^r \rho_i \boxtimes S_{a_i} \right) 
\oplus \phi_0 \oplus \left(\bigoplus_{i=1}^r \rho_i^\vee \boxtimes S_{a_i} \right) 
\]
with $e(\rho_1) \leq \dots \leq e(\rho_r) < 0$ and $\phi_0 \in \Phi_\temp(G_0)$, 
then for $\ep \in \widehat{A_\phi} = \widehat{A_{\phi_0}}$, 
we define $\pi(\phi, \ep)$ as the unique irreducible subrepresentation of 
\[
\Delta_{\rho_1}[x_1,-x_1] \times \dots \times \Delta_{\rho_r}[x_r,-x_r] \rtimes \pi(\phi_0,\ep)
\]
with $x_i = \half{a_i-1}$.
We also write it as 
\[
\pi(\phi, \ep) = L(\mm; \pi(\phi_0,\ep)) 
= L(\Delta_{\rho_1}[x_1,-x_1], \dots, \Delta_{\rho_r}[x_r,-x_r]; \pi(\phi_0,\ep)), 
\]
where $\mm = [x_1,-x_1]_{\rho_1} + \dots + [x_r,-x_r]_{\rho_r}$.
Then $\Pi_\phi$ is a subset of $\Irr(G)$, 
the map $\Pi_\phi \rightarrow \widehat{\AA_\phi}$ given by $\pi(\phi, \ep) \mapsto \ep$ is bijective, 
and 
\[
\Irr(G) = \bigsqcup_{\phi \in \Phi(G)} \Pi_\phi.
\]
We call $\Pi_\phi$ the \emph{$L$-packet} associated to $\phi$. 
\par

Let $\phi \in \Phi(G)$. 
As we have seen above, we can decompose it as 
$\phi = \phi_\good \oplus \phi_\bad$ 
with $\phi_\good$ (\resp $\phi_\bad$) of good (\resp bad) parity. 
If $\pi = \pi(\phi, \ep)$, we set $\pi_\good = \pi(\phi_\good, \ep)$. 
We say that $\pi$ is \emph{of good parity} if $\pi = \pi_\good$, or equivalently, $\phi = \phi_\good$.

%\subsection{Geometric Lemma}
\subsection{Geometric Lemma}
We recall Geometric Lemma and Tadi\'c's formula in this subsection. 
\par

Let $G$ be either $\SO_{2n+1}(F)$ or $\Sp_{2n}(F)$. 
Fix an $F$-rational Borel subgroup $B = TU$ of $G$. 
Let $W$ be the Weyl group of $G$. 
For two standard parabolic subgroup $P_1 = M_1N_1$ and $P_2 = M_2N_2$, 
set 
\[
W^{M_1,M_2} = \{w \in W \;|\; w(M_1 \cap B)w^{-1} \subset B,\; w^{-1}(M_2 \cap B)w \subset B\}. 
\]
For any $w \in W^{M_1,M_2}$, we define a functor
\[
F_w \colon \Rep(M_1) \rightarrow \Rep(M_2)
\]
by 
\[
F_w = \Ind_{wP_1w^{-1} \cap M_2}^{M_2} \circ \Ad(w) \circ \Jac_{w^{-1}P_2w \cap M_1}.
\]

\begin{thm}[Geometric Lemma ({\cite[2.11]{BZ}})]\label{GeometricLemma}
The functor $F = \Jac_{P_2} \circ \Ind_{P_1}^G \colon \Rep(M_1) \rightarrow \Rep(M_2)$
is glued from functors $F_w$, for $w \in W^{M_1,M_2}$.
\end{thm}

In \cite{T-str}, Tadi\'c studied $W^{M_1,M_2}$ when $P_1$ and $P_2$ are maximal, 
and he got the following formula after  semisimplification. 

\begin{cor}[Tadi\'c's formula {\cite[Theorems 5.4, 6.5]{T-str}}]\label{TadicFormula}
Suppose that $P = MN$ is a maximal parabolic subgroup of $G$ with $M \cong \GL_m(F) \times G_0$.
Then, for any finite length representation of $G$ of the form $\tau \rtimes \pi$, we have
\[
[\Jac_P(\tau \rtimes \pi)] = \sum_{\substack{n_1,n_2,n_3,n_4 \geq 0 \\ n_1+n_3+n_4 = m}}
\sum_{\substack{\tau_1 \otimes \tau_2 \otimes \tau_3 \\ \tau_4 \otimes \pi_0}}
\left(\tau_1 \times \tau_4 \times \tau_3^\vee\right) \otimes \left( \tau_2 \rtimes \pi_0 \right), 
\]
where 
\begin{itemize}
\item
$\tau_1 \otimes \tau_2 \otimes \tau_3$ runs over irreducible representations appearing in 
$[\Jac_R(\tau)]$ where $R$ is the standard parabolic subgroup  corresponding to the partition $(n_1,n_2,n_3)$;
\item
$\tau_4 \otimes \pi_0$ runs over irreducible representations appearing in 
$[\Jac_{P'}(\pi_0)]$ where $P'$ is the standard parabolic subgroup whose Levi is of the form $\GL_{n_4}(F) \times G_0'$.
\end{itemize}
Here, when $n_i = 0$, we understand that $\tau_i = \1_{\GL_0(F)}$.
\end{cor}

%\subsection{Methods for proving (non-)unitarity}
\subsection{Methods for proving (non-)unitarity}
Recall that an irreducible representation $\pi$ of a group $G$ is \emph{hermitian}
if $\overline{\pi} \cong \pi^\vee$, 
where $\overline{\pi}$ denotes the complex conjugate of $\pi$.
In other words, $\pi$ is hermitian if and only if it admits a non-degenerate $G$-invariant hermitian form. 
Note that if $\pi$ is unitary, then it is also hermitian.
\par

Now we list some basic methods for proving or disproving unitarity. 

\begin{prop}\label{list}
Fix a standard parabolic subgroup $P=MN$ of $G$ 
with $M \cong \GL_{m_1}(F) \times \dots \times \GL_{m_k}(F) \times G_0$. 
Let $\tau_i \in \Irr(\GL_{m_i}(F))$, $\pi_0 \in \Irr(G_0)$ 
and set $\pi_M = \tau_1 \boxtimes \dots \boxtimes \tau_k \boxtimes \pi_0 \in \Irr(M)$.
\par

\begin{enumerate}
\item(Unitary induction (UI))
If $\pi_M$ is unitary, 
then $\Ind_P^G(\pi_M)$ is a direct sum of irreducible unitary representations of $G$. 

\item(Unitary reduction (UR))
If $\pi_M$ is hermitian and if $\Ind_P^G(\pi_M)$ is irreducible and unitary, 
then $\pi_M$ is also unitary. 

\item(Complimentary series (CS))
Let $x_1(t), \dots, x_r(t) \colon [0,1] \rightarrow \R$ be continuous functions, 
and set 
\[
\Pi_t = \tau_1|\cdot|^{x_1(t)} \times \dots \times \tau_r|\cdot|^{x_r(t)} \rtimes \pi_0.
\]
If $\Pi_t$ is irreducible and hermitian for $0 \leq t < 1$, and if $\Pi_0 = \Ind_P^G(\pi_M)$ is unitary, 
then all irreducible subquotients of $\Pi_t$ are unitary for $0 \leq t \leq 1$. 

\item(Beyond the first reducibility point (RP1))
Suppose that 
\begin{itemize}
\item
$k=1$; 
\item
$\tau_1 = \times_{i=1}^r\Sp(\rho_i,c_i,d_i)$ is a product of unitary Speh representations
with $\rho_i \cong \rho_i^\vee$ for $1 \leq i \leq r$; 
\item
$\pi_0$ is of Arthur type of good parity.
\end{itemize}
Let $\psi_M \in \Psi(M)$ be an $A$-parameter with $\pi_M \in \Pi_{\psi_M}$, 
and let $R_P(w,\pi_M,\psi_M)$ be the normalized intertwining operator defined by Arthur \cite[Section 2.4]{Ar} 
with $w \in W(M,G)$. 
Assume further that $R_P(w,\pi_M,\psi_M)$ is not a scalar (so in particular $\Ind_P^G(\pi_M)$ is reducible).
Then $\tau|\cdot|^s \rtimes \pi_0$ is not unitary for sufficiently small  $s > 0$. 
\end{enumerate}
\end{prop}
For (UI) and (UR) (\resp (CS)), see \cite{T-ext} or \cite[Section 2]{MT} (\resp \cite[Lemma 3.3]{LMT}). 
On the other hand, (RP1) was established in the latter part of \cite[Section 2]{MT}, 
where another criterion (RP2) was also presented.

%\section{Main Theorem}
%\section{Main Theorem}
\section{Main Theorem}
Let $n \geq 0$, and fix $G$ to be either the split group $\SO_{2n+1}(F)$ or $\Sp_{2n}(F)$. 
Accordingly, define
\[
\Irr^G = 
\begin{cases}
\bigcup_{m \geq 0} \Irr(\SO_{2m+1}(F)) & \text{if } G = \SO_{2n+1}(F), \\
\bigcup_{m \geq 0} \Irr(\Sp_{2m}(F)) & \text{if } G = \Sp_{2n}(F),
\end{cases}
\]
and set
\[
\Irr^\GL = \bigcup_{m \geq 0} \Irr(\GL_m(F)).
\]

%\subsection{Statement}
\subsection{Statement}
Now we can state our main theorem. 
\begin{thm}\label{unitary_gp}
If $\pi(\phi,\ep)$ is an irreducible unitary representation of $G$, 
then $\pi(\phi_\good, \ep)$ is of Arthur type. 
\end{thm}

Note that one can determine whether $\pi(\phi_\good, \ep)$ is of Arthur type 
using the algorithm described in \cite{At-Is_Arthur}. 
Since every irreducible representation of Arthur type is known to be unitary, 
we obtain the following corollary.

\begin{cor}
Let $\pi \in \Irr(G)$ be of good parity. 
Then $\pi$ is unitary if and only if $\pi$ is of Arthur type.
\end{cor}

%\subsection{The SZ-Decomposition}
\subsection{The SZ-Decomposition}\label{sec.SZ}
By the Langlands classification, 
every $\pi(\phi,\ep) \in \Irr^G$ can be written as 
\[
\soc\left(\tau_\bad^- \times 
\Delta_{\rho_1}[x_1,y_1]^{k_1} \times \dots \times \Delta_{\rho_t}[x_t,y_t]^{k_t} 
\rtimes \pi_\temp\right), 
\]
where 
\begin{itemize}
\item
$\tau_\bad^-$ is an irreducible representation of some general linear group
whose $L$-parameter $\phi_\bad^-$ satisfies that $\phi_\bad = \phi_\bad^- \oplus (\phi_\bad^-)^\vee$; 
\item
$\pi_\temp \in \Irr^G$ is tempered; 
\item
$\rho_i \in \Cusp^\GL$ is unitary; 
\item
$x_i+y_i < 0$ and $x_1 \leq \dots \leq x_t$;
\item
If $\rho_i \cong \rho_j$, then 
\[
i < j \iff x_i < x_j, \text{ or } x_i = x_j \text{ and } y_i < y_j.
\]
\end{itemize}
Here, $\Delta_{\rho_i}[x_i,y_i]^{k_i} = \Delta_{\rho_i}[x_i,y_i] \times \dots \times \Delta_{\rho_i}[x_i,y_i]$ ($k_i$ times).
In particular, $\pi(\phi_\good, \ep) = L(\Delta_{\rho_1}[x_1,y_1]^{k_1}, \dots, \Delta_{\rho_t}[x_t,y_t]^{k_t}; \pi_\temp)$.
Notice that $\tau_\bad^- \times \Delta_{\rho_i}[x_i,y_i]$ is irreducible by \cite[Proposition 8.6]{Z}.
\par

Write $\{i \;|\; x_i < 0\} = \{1, \dots, s\}$, and set $\tau_i^- = \Delta_{\rho_i}[x_i,y_i]^{k_i}$.
If we set 
\[
\pi' = L(\Delta_{\rho_{s+1}}[x_{s+1},y_{s+1}]^{k_{s+1}}, \dots, \Delta_{\rho_t}[x_t,y_t]^{k_t}; \pi_\temp),
\]
then $\pi$ is the socle of 
$\tau_\bad^- \times \tau_1^- \times \dots \times \tau_s^- \rtimes \pi'$.
Note that for $\rho \in \Cusp^\GL$ with $\rho^\vee \cong \rho$ and $x \in \R$, 
if $D_{\rho|\cdot|^x}(\pi') \not= 0$, then $x \geq 0$.
Next, let $\hat\pi'$ be the Aubert dual of $\pi'$ (see \cite{Au}).
We write
\[
\hat\pi' = L(\Delta_{\rho'_1}[x'_1,y'_1]^{k'_1}, \dots, \Delta_{\rho'_{t'}}[x'_{t'},y'_{t'}]^{k'_{t'}}; \pi'_\temp)
\]
as above. 
Writing $\{i \;|\; x'_i \leq -1\} = \{1, \dots, r\}$, 
we can find $\pi_0 \in \Irr^G$ such that 
\[
\hat\pi' = \soc(\Delta_{\rho'_1}[x'_1,y'_1]^{k'_1} \times \dots \times \Delta_{\rho'_{r}}[x'_{r},y'_{r}]^{k'_{r}} 
\rtimes \hat\pi_0).
\]
Note that for $\rho \in \Cusp^\GL$ with $\rho^\vee \cong \rho$ and $x \in \R$, 
if $D_{\rho|\cdot|^x}(\hat\pi_0) \not= 0$, then $-\half{1} \leq x \leq 0$.
If we set $\tau_i^+ = Z_{\rho_i'}[-x'_i,-y'_i]^{k'_i}$, 
we conclude that $\pi$ is the socle of 
\[
\tau_\bad^- \times \tau_1^- \times \dots \times \tau_s^- \times \tau_1^+ \times \dots \times \tau_r^+ \rtimes \pi_0.
\]
Moreover, 
if we set 
\begin{align*}
&\pi_1 = \soc(\tau_r^+ \rtimes \pi_0), \dots, \pi_r = \soc(\tau_1^+ \rtimes \pi_{r-1}), \\
&\pi_{r+1} = \soc(\tau_s^- \rtimes \pi_r), \dots, \pi_{r+s} = \soc(\tau_1^- \rtimes \pi_{r+s-1}), 
\end{align*}
then we see that  
\begin{itemize}
\item
$\pi_i$ is irreducible for all $0 \leq i \leq r+s$ with $\pi(\phi_\good, \ep) = \pi_{r+s}$; 
\item
$\pi_0$ satisfies that 
\begin{talign*}
D_{\rho|\cdot|^x}(\pi_0) \not= 0 \implies x \in \{0,\frac{1}{2}\}
\end{talign*}
for any $\rho \in \Cusp^\GL$ with $\rho^\vee \cong \rho$; 
\item
for $1 \leq i \leq r$, we have $\pi_i = \soc(\tau_{r-i+1}^+ \rtimes \pi_{i-1})$ and 
\begin{talign*}
D_{\rho|\cdot|^x}(\pi_i) \not= 0 \implies 0 \leq x \leq -x'_{r-i+1}
\end{talign*}
for any $\rho \in \Cusp^\GL$ with $\rho^\vee \cong \rho$; 
\item
for $r+1 \leq i \leq r+s$, we have $\pi_i = \soc(\tau_{r+s-i+1}^- \rtimes \pi_{i-1})$ and 
\begin{talign*}
D_{\rho|\cdot|^x}(\pi_i) \not= 0 \implies x \geq x_{r+s-i+1}
\end{talign*}
for any $\rho \in \Cusp^\GL$ with $\rho^\vee \cong \rho$.
\end{itemize}
Here, we note that $x_1 \leq \dots \leq x_s < 0$ and $x'_1 \leq \dots \leq x'_r \leq -1$.
We call 
\[
(\tau_\bad^-, \{\tau_j^-\}_{1 \leq j \leq s}, \{\tau_j^+\}_{1 \leq j \leq r}, \{\pi_j\}_{0 \leq j \leq r+s})
\]
the \emph{SZ-decomposition} of $\pi$.
\par

The following is the inductive step of the proof of Theorem \ref{unitary_gp}.
\begin{thm}\label{inductive}
Let $\pi(\phi,\ep) \in \Irr^G$, 
and let $(\tau_\bad^-, \{\tau_j^-\}_{1 \leq j \leq s}, \{\tau_j^+\}_{1 \leq j \leq r}, \{\pi_j\}_{0 \leq j \leq r+s})$
be its SZ-decomposition.
For $0 \leq j < r+s$, if $\pi(\phi,\ep)$ is unitary and $\pi_j$ is of Arthur type, 
then $\pi_{j+1}$ is of Arthur type. 
In particular, if $\pi(\phi,\ep)$ is unitary and $\pi_0$ is of Arthur type, then $\pi(\phi_\good,\ep)$ is of Arthur type.
\end{thm}

Sections \ref{sec.observe}--\ref{sec.proof_inductive}
are devoted to proving Theorem \ref{inductive}. 

%\subsection{A key observation}
\subsection{A key observation}\label{sec.observe}
We begin with a first observation.
\begin{lem}\label{pi'}
Let $\pi(\phi,\ep) \in \Irr^G$, and let 
\[
(\tau_\bad^-, \{\tau_j^-\}_{1 \leq j \leq s}, \{\tau_j^+\}_{1 \leq j \leq r}, \{\pi_j\}_{0 \leq j \leq r+s})
\]
be its SZ-decomposition. Assume that $\pi(\phi,\ep)$ is unitary.

\begin{enumerate}
\item
Assume $s > 0$ and fix $r \leq j < r + s$. 
Write $\tau_{r+s-j}^- = \Delta_\rho[B-1,-A-1]^k$ 
for some $B < 1$ and $A > 0$ with $A + B \geq 0$, where $\rho \in \Cusp^\GL$ is unitary.
Let $S$ be the product of $k$ copies of the Speh representation
\[
\begin{pmatrix}
B & \cdots & A \\
\vdots & \ddots & \vdots \\
-A & \cdots & -B
\end{pmatrix}_\rho.
\]
Then there exists an irreducible representation $\pi'_{j+1}$ such that:
\begin{itemize}
\item
$\pi'_{j+1} \hookrightarrow S \rtimes \pi_{j+1}$; and
\item
$\pi'_{j+1} \hookrightarrow \soc(\tau_{r+s-j}^- \times S) \rtimes \pi_j$.
\end{itemize}

\item
Assume $r > 0$ and fix $0 \leq j < r$. 
Write $\tau_{r-j}^+ = Z_\rho[B+1,A+1]^k$ 
for some $A \geq B \geq 0$, where $\rho \in \Cusp^\GL$ is unitary.
Let $S$ be the product of $k$ copies of the Speh representation
\[
\begin{pmatrix}
B & \cdots & A \\
\vdots & \ddots & \vdots \\
-A & \cdots & -B
\end{pmatrix}_\rho.
\]
Then there exists an irreducible representation $\pi'_{j+1}$ such that:
\begin{itemize}
\item
$\pi'_{j+1} \hookrightarrow S \rtimes \pi_{j+1}$; and
\item
$\pi'_{j+1} \hookrightarrow \soc(\tau_{r-j}^+ \times S) \rtimes \pi_j$.
\end{itemize}
\end{enumerate}
\end{lem}
\begin{proof}
We only prove (2) since the proof of (1) is similar. 
(In fact, the proof of (2) essentially includes the one of (1).)
\par

We fix $0 \leq j < r$. 
Recall that we have inclusions 
\begin{align*}
\pi(\phi,\ep)
&\hookrightarrow \tau_\bad^- \rtimes \pi_{r+s} 
\\&\hookrightarrow \tau_1^- \times \tau_\bad^- \rtimes \pi_{r+s-1} 
\\ &\quad\vdots \\
&\hookrightarrow \tau_1^- \times \dots \times \tau_s^- \times \tau_\bad^- \rtimes \pi_r 
\\&\hookrightarrow \tau_1^- \times \dots \times \tau_s^- \times \tau_1^+ \times \tau_\bad^- \rtimes \pi_{r-1} 
\\ &\quad\vdots 
\\&\hookrightarrow 
\tau_1^- \times \dots \times \tau_s^- \times \tau_1^+ \times \dots \times \tau_r^+ 
\times \tau_\bad^- \rtimes \pi_0.
\end{align*}
Consider the unitary induction $S \rtimes \pi(\phi,\ep)$.
We take an irreducible subquotient $\pi''_{r+s}$ of $S \rtimes \pi(\phi,\ep)$ such that 
\[
\tag{$\ast$}
[\Jac(\pi''_{r+s})] \geq 
\tau_1^- \otimes \dots \otimes \tau_s^- \otimes \tau_1^+ \otimes \dots \otimes \tau_r^+ 
\otimes S \otimes \tau_\bad^- \otimes \pi_0.
\]
Since $\pi(\phi,\ep)$ is unitary, so is $S \rtimes \pi(\phi,\ep)$, and hence it is semisimple.
Therefore $\pi''_{r+s} \hookrightarrow S \rtimes \pi(\phi,\ep)$. 
\par

Recall that $\tau_1^- \times S$ is SI by \cite[Corollary 4.10]{LM}.
We have 
\begin{align*}
[\pi''_{r+s}] 
&\leq [S \times \tau_\bad^- \rtimes \pi_{r+s}] 
\\&\leq [S \times \tau_1^- \times \tau_\bad^- \rtimes \pi_{r+s-1}]
\\&= [\tau_1^- \times S \times \tau_\bad^- \rtimes \pi_{r+s-1}]
\\&= [\soc(\tau_1^- \times S) \times \tau_\bad^- \rtimes \pi_{r+s-1}]
+\left[\frac{\tau_1^- \times S}{\soc(\tau_1^- \times S)} \times \tau_\bad^- \rtimes \pi_{r+s-1}\right].
\end{align*}
Note that
\[
\left[\Jac\left(\frac{\tau_1^- \times S}{\soc(\tau_1^- \times S)}\right)\right]
\]
has no irreducible subquotient of the form $\tau_1^- \otimes S'$ for any $S'\not= 0$.
Recall that we denote $\tau_i^- = \Delta_{\rho_i}[x_i,y_i]^{k_i}$ 
and that $x_1 \leq \dots \leq x_s $ are all negative.
Since $B \geq 0$, by construction, 
we see that 
\[
\left[\Jac\left(\frac{\tau_1^- \times S}{\soc(\tau_1^- \times S)} \times \tau_\bad^- \rtimes \pi_{r+s-1}\right)\right]
\]
has no irreducible subquotient of the form 
$\tau_1^- \otimes \dots \otimes \tau_s^- \otimes \tau_1^+ \otimes \dots \otimes \tau_{r}^+ 
\otimes S \otimes \sigma$ for any $\sigma \not= 0$.
Hence the composition map
\[
\pi''_{r+s} \hookrightarrow S \times \tau_\bad^- \rtimes \pi_{r+s} 
\hookrightarrow S \times \tau_1^- \times \tau_\bad^- \rtimes \pi_{r+s-1}
\]
factors through 
\[
\pi''_{r+s} 
\hookrightarrow \soc(\tau_1^- \times S) \times \tau_\bad^- \rtimes \pi_{r+s-1} 
\hookrightarrow \tau_1^- \times S \times \tau_\bad^- \rtimes \pi_{r+s-1}. 
\]
In particular, there is an irreducible subquotient $\pi''_{r+s-1}$ of $S \times \tau_\bad^- \rtimes \pi_{r+s-1}$ 
such that 
\[
\pi''_{r+s} \hookrightarrow \tau_1^- \rtimes \pi''_{r+s-1}.
\] 
By construction, the Langlands data of $\pi''_{r+s}$ are obtained by 
adding $\tau_1^- = \Delta_{\rho_1}[x_1, y_1]^{k_1}$ to those of $\pi''_{r+s-1}$. 
Thus, $\pi''_{r+s-1}$ is uniquely characterized by the condition
\[
[\Jac(\pi''_{r+s})] = \tau_1^- \otimes (m \cdot \pi''_{r+s-1}) 
+ \sum_{\substack{\tau \in \Irr(\GL_{n_1}(F)) \\ \tau \not\cong \tau_1^-}} 
\tau \otimes \Pi_\tau,
\]
for some representation $\Pi_\tau$ and some multiplicity $m > 0$, 
where $n_1$ is such that $\tau_1^- \in \Irr(\GL_{n_1}(F))$.
By applying Frobenius reciprocity to the inclusion map $\pi''_{r+s} \hookrightarrow \tau_1^- \times S \times \tau_\bad^- \rtimes \pi_{r+s-1}$, 
we have a nonzero equivariant map
\[
\tau_1^- \boxtimes \pi''_{r+s-1} \rightarrow \tau_1^- \boxtimes (S \times \tau_\bad^- \rtimes \pi_{r+s-1}).
\]
Since $\pi''_{r+s-1}$ is irreducible, 
we obtain an inclusion $\pi''_{r+s-1} \hookrightarrow S \times \tau_\bad^- \rtimes \pi_{r+s-1}$.
On the other hand, by $(\ast)$, 
we see that 
\[
[\Jac(\pi''_{r+s-1})] \geq 
\tau_2^- \otimes \dots \otimes \tau_s^- \otimes \tau_1^+ \otimes \dots \otimes \tau_r^+ 
\otimes S \otimes \tau_\bad^- \otimes \pi_0.
\]
\par

Repeating this argument, for $0 \leq i \leq s$, 
there is an irreducible representation $\pi''_{r+s-i}$ such that 
\begin{itemize}
\item
$\pi''_{r+s-i} \hookrightarrow S \times \tau_\bad^- \rtimes \pi_{r+s-i}$; and
\item
$[\Jac(\pi''_{r+s-i})] \geq 
\underbrace{\tau_{i+1}^- \otimes \dots \otimes \tau_s^-}_{s-i} 
\otimes \tau_1^+ \otimes \dots \otimes \tau_r^+ 
\otimes S \otimes \tau_\bad^- \otimes \pi_0$.
\end{itemize}
In particular, considering $i=s$, we have an irreducible representation $\pi''_r$ such that 
\begin{itemize}
\item
$\pi''_r \hookrightarrow S \times \tau_\bad^- \rtimes \pi_r$; and
\item 
$[\Jac(\pi''_r)] \geq 
\tau_1^+ \otimes \dots \otimes \tau_{r}^+ \otimes S \otimes \tau_\bad^- \otimes \pi_0.$
\end{itemize}
\par

By the same argument, 
we have 
\begin{align*}
[\pi''_r] 
&\leq [S \times \tau_\bad^- \rtimes \pi_r] 
\\&\leq [S \times \tau_1^+ \times \tau_\bad^- \rtimes \pi_{r-1}]
\\&= [\tau_1^+ \times S \times \tau_\bad^- \rtimes \pi_{r-1}]
\\&= [\soc(\tau_1^+ \times S) \times \tau_\bad^- \rtimes \pi_{r-1}]
+\left[\frac{\tau_1^+ \times S}{\soc(\tau_1^+ \times S)} \times \tau_\bad^- \rtimes \pi_{r-1}\right].
\end{align*}
Recall that $\tau_i^+ = Z_{\rho_i'}[-x'_i,-y'_i]^{k'_i}$ 
with $-x_1' \geq \dots \geq -x'_{r-j} = B+1 > B$.
Hence 
\[
\Jac\left(\frac{\tau_1^+ \times S}{\soc(\tau_1^+ \times S)} \rtimes \pi_{r-1}\right)
\]
has no irreducible representation of the form 
$\tau_1^+ \otimes \dots \otimes \tau_r^+ \otimes S \otimes \sigma$ 
for any $\sigma \not= 0$.
Hence the composition map 
\[
\pi''_r \hookrightarrow S \times \tau_\bad^- \rtimes \pi_r 
\hookrightarrow S \times \tau_1^+ \times \tau_\bad^- \rtimes \pi_{r-1}
\]
factors through 
\[
\pi''_r \hookrightarrow \soc(\tau_1^+ \times S) \times \tau_\bad^- \rtimes \pi_{r-1}
\hookrightarrow \tau_1^+ \times S \times \tau_\bad^- \rtimes \pi_{r-1}, 
\]
and we can find an irreducible subquotient $\pi''_{r-1}$ of $S \times \tau_\bad^- \rtimes \pi_{r-1}$ such that 
\[
\pi''_r \hookrightarrow \tau_1^+ \rtimes \pi''_{r-1}. 
\]
By applying Frobenius reciprocity to 
$\pi''_r \hookrightarrow \tau_1^+ \times S \times \tau_\bad^- \rtimes \pi_{r-1}$, 
we obtain an inclusion $\pi''_{r-1} \hookrightarrow S \times \tau_\bad^- \rtimes \pi_{r-1}$.
Moreover, $\pi''_{r-1}$ satisfies that 
\[
[\Jac(\pi''_{r-1})] \geq 
\tau_2^+ \otimes \dots \otimes \tau_r^+ \otimes S \otimes \tau_\bad^- \otimes \pi_0. 
\]
\par

This argument can be repeated for $0 \leq i \leq r-j-1$,  
and we obtain an irreducible representation $\pi''_{r-i}$ such that 
\begin{itemize}
\item
$\pi''_{r-i} \hookrightarrow S \times \tau_\bad^- \rtimes \pi_{r-i}$; and
\item
$[\Jac(\pi''_{r-i})] \geq 
\underbrace{\tau_{i+1}^+ \otimes \dots \otimes \tau_{r}^+}_{r-i} 
\otimes S \otimes \tau_\bad^- \otimes \pi_0.$
\end{itemize}
Moreover, by the argument for $i=r-j-1$, 
we see that 
\[
\pi''_{j+1} \hookrightarrow S \times \tau_\bad^- \rtimes \pi_{j+1} 
\hookrightarrow S \times \tau_{r-j}^+ \times \tau_\bad^- \rtimes \pi_j
\]
factors through $\pi''_{j+1} \hookrightarrow \soc(\tau_{r-j}^+ \times S) \times \tau_\bad^- \rtimes \pi_j$. 
\par

Finally, let $\pi'_{j+1} \in \Irr^G$ be such that $\pi''_{j+1} = \soc(\tau_\bad^- \rtimes \pi'_{j+1})$. 
Applying Frobenius reciprocity again to the embeddings 
\[
\pi''_{j+1} \hookrightarrow \tau_\bad^- \times S \rtimes \pi_{j+1} 
\quad \text{and} \quad 
\pi''_{j+1} \hookrightarrow \tau_\bad^- \times \soc(\tau_{r-j}^+ \times S) \rtimes \pi_j,
\]
we obtain inclusions 
\[
\pi'_{j+1} \hookrightarrow S \rtimes \pi_{j+1} 
\quad \text{and} \quad 
\pi'_{j+1} \hookrightarrow \soc(\tau_{r-j}^+ \times S) \rtimes \pi_j.
\]
This completes the proof of Lemma~\ref{pi'}.
\end{proof}

The existence of the representation $\pi'_{j+1}$ in Lemma~\ref{pi'} will be used to show that $\pi_{j+1}$ is of Arthur type, assuming that $\pi_j$ is of Arthur type. 
To simplify the notation in what follows, we denote 
\[
\pi_j = \sigma, \quad \pi_{j+1} = \pi, \quad \pi'_{j+1} = \pi'.
\]
Theorem~\ref{inductive} is thus reduced to the following proposition.

\begin{prop}\label{typeA}
Let $\sigma, \pi, \pi' \in \Irr^G$ be of good parity. 
Assume that $\sigma$ is of Arthur type. 
\begin{enumerate}
\item
Fix half-integers $B < 1$ and $A > 0$ such that $A+B \geq 0$. 
Set $\tau = \Delta_\rho[B-1,-A-1]^k$
and
let $S$ be the product of $k$ copies of the Speh representation
\[
\begin{pmatrix}
B & \cdots & A \\
\vdots & \ddots & \vdots \\
-A & \cdots & -B
\end{pmatrix}_\rho,
\]
where $\rho \in \Cusp^\GL$ is unitary.
Suppose that 
\begin{itemize}
\item
if $D_{\rho|\cdot|^x}(\sigma) \not= 0$, then $x \geq B-1$, 
and $D_{\rho|\cdot|^{-A-1}} \circ \dots \circ D_{\rho|\cdot|^{B-1}}(\sigma) = 0$; 
\item
$\pi = \soc(\tau \rtimes \sigma)$;
\item
$\pi' \hookrightarrow S \rtimes \pi$ and $\pi' \hookrightarrow \soc(\tau \times S) \rtimes \sigma$, 
i.e., 
\[
S \rtimes \soc(\tau \rtimes \sigma)
\quad\text{and}\quad
\soc(\tau \times S) \rtimes \sigma
\]
have a common irreducible subrepresentation. 
\end{itemize}
Then $\pi$ is also of Arthur type. 

\item
Fix half-integers $A \geq B > -1$. 
Set $\tau = Z_\rho[B+1,A+1]^k$ 
and
let $S$ be the product of $k$ copies of the Speh representation
\[
\begin{pmatrix}
B & \cdots & A \\
\vdots & \ddots & \vdots \\
-A & \cdots & -B
\end{pmatrix}_\rho,
\]
where $\rho \in \Cusp^\GL$ is unitary.
Suppose that 
\begin{itemize}
\item
if $D_{\rho|\cdot|^x}(\sigma) \not= 0$, then $x \leq B+1$, 
and $D_{\rho|\cdot|^{A+1}} \circ \dots \circ D_{\rho|\cdot|^{B+1}}(\sigma) = 0$; 
\item
$\pi = \soc(\tau \rtimes \sigma)$;
\item
$\pi' \hookrightarrow S \rtimes \pi$ and $\pi' \hookrightarrow \soc(\tau \times S) \rtimes \sigma$, 
i.e., 
\[
S \rtimes \soc(\tau \rtimes \sigma)
\quad\text{and}\quad
\soc(\tau \times S) \rtimes \sigma
\]
have a common irreducible subrepresentation. 
\end{itemize}
Then $\pi$ is also of Arthur type. 
\end{enumerate}
\end{prop}

The proof of Proposition \ref{typeA} will be completed in Section \ref{sec.proof_inductive}.
The next subsection is the critical step of the proof. 

%\subsection{Inequalities of Jacquet modules}
\subsection{Inequalities of Jacquet modules}\label{sec.inequalities}
To attack Proposition \ref{typeA}, in this subsection, 
we establish some inequalities of Jacquet modules. 
\par

Define a sequence $\{x_j\}_{1 \leq j \leq ab}$ with $a = A+B+1$ and $b = A-B+1$ by
\[
\Sp(\rho,a,b) = 
\begin{pmatrix}
B & \ldots & A\\
\vdots & \ddots & \vdots \\
-A & \ldots & -B
\end{pmatrix}_\rho
=
\begin{pmatrix}
x_{1} & \ldots & x_{a(b-1)+1}\\
x_{2} & \ldots & x_{a(b-1)+2}\\
\vdots & \ddots & \vdots \\
x_{a} & \ldots & x_{ab}
\end{pmatrix}_\rho.
\]
Hence $-A \leq x_j \leq A$ for $1 \leq j \leq ab$.
To $S = \Sp(\rho,a,b)^k$, we attach an operator $D_S$ defined by
\[
D_S = D_{\rho|\cdot|^{x_{ab}}}^{(k)} \circ \dots \circ D_{\rho|\cdot|^{x_1}}^{(k)}.
\]
\par

\begin{rem}\label{Max}
One might (wrongly) expect that the operator $D_S$ distinguishes irreducible representations $\tau$ 
containing $S \otimes \tau'$ for some $\tau' \in \Irr^G$, 
in the sense that $D_S(\tau) \neq 0$ if and only if $S \otimes \tau' \leq [\Jac(\tau)]$ for some $\tau' \not=0$. 
However, this is not the case, even when $k = 1$. 
We are grateful to Max Gurevich for providing us with a counterexample to this assertion.
\par

For simplicity, let $\rho = \1_{\GL_1(F)}$. 
Let $S = \Sp(\rho,4,4)$ 
and consider the multisegment $\mathfrak{m} = [-1,-2]_\rho + [0,0]_\rho + [1,-1]_\rho + [2,1]_\rho$. 
The representation $L(\mathfrak{m}) \in \Irr(\GL_8(F))$ was studied by B. Leclerc \cite{Lec}, 
and is an example of what is called in \cite{LM2} a \emph{non-square irreducible representation}.
\par

We claim that $D_S(L(\mathfrak{m}|\cdot|^{-1} + \mathfrak{m}|\cdot|^1)) \neq 0$. 
Computer-based calculations show that: 
\begin{itemize}
\item
$D_S(L(\mathfrak{m}|\cdot|^{-1}) \times L(\mathfrak{m}|\cdot|^1)) = \C^2$; 
\item
$L(\mathfrak{m}|\cdot|^{-1}) \times L(\mathfrak{m}|\cdot|^1)$ has length $4$; 
\item
$S$ appears with multiplicity one in $L(\mathfrak{m}|\cdot|^{-1}) \times L(\mathfrak{m}|\cdot|^1)$; 
\item
For the two irreducible representations 
$\tau_1, \tau_2 \leq [L(\mathfrak{m}|\cdot|^{-1}) \times L(\mathfrak{m}|\cdot|^1)]$
other than $S$ and $L(\mathfrak{m}|\cdot|^{-1} + \mathfrak{m}|\cdot|^1)$, 
we have $D_S(\tau_1) = D_S(\tau_2) = 0$.
\end{itemize}
Hence we conclude that $D_S(L(\mathfrak{m}|\cdot|^{-1} + \mathfrak{m}|\cdot|^1)) \neq 0$. 
%Indeed, a direct computation shows that
%\[
%D_S(L(\mathfrak{m}|\cdot|^{-1}) \times L(\mathfrak{m}|\cdot|^1)) = \C^2.
%\]
%Since $S$ appears with multiplicity one in $L(\mathfrak{m}|\cdot|^{-1}) \times L(\mathfrak{m}|\cdot|^1)$, 
%there must exist some other irreducible subquotient $\tau$ of this induced representation 
%such that $D_S(\tau) \neq 0$. 
%As the induced representation has length $4$, 
%a direct computation shows that 
%this subquotient has to be $L(\mathfrak{m}|\cdot|^{-1} + \mathfrak{m}|\cdot|^1)$.
\end{rem}

Using the operator $D_S$, we state the following lemma. 
This is the most technical part of the proof of Theorem \ref{unitary_gp}.
It appears naturally in M{\oe}glin's explicit construction of $A$-packets (see Section \ref{sec.Moe}).

\begin{lem}\label{geometric}
Let $\sigma, \pi, \pi' \in \Irr^G$ be of good parity. 
Assume that $\sigma$ is of Arthur type. 
\begin{enumerate}
\item
Suppose that we are in the situation of Proposition \ref{typeA} (1). 
Fix integers $t_1, \dots, t_k > 1$, and  set 
\[
S' = \bigtimes_{i=1}^k
\begin{pmatrix}
B-t_i & \ldots & B-2 \\
\vdots & \ddots & \vdots \\
-A-t_i & \ldots &-A-2
\end{pmatrix}_\rho.
\]
Then 
\[
D_S\left(\soc(\soc(S' \times \tau \times S) \rtimes \sigma)\right)
\geq \soc(S' \rtimes \pi).
\]

\item
Suppose that we are in the situation of Proposition \ref{typeA} (2). 
Fix integers $t_1, \dots, t_k > 1$, and  set 
\[
S' = \bigtimes_{i=1}^k
\begin{pmatrix}
B+t_i & \ldots & A+t_i \\
\vdots & \ddots & \vdots \\
B+2 & \ldots &A+2
\end{pmatrix}_\rho.
\] 
Then 
\[
D_S\left(\soc(\soc(S' \times \tau \times S) \rtimes \sigma)\right)
\geq \soc(S' \rtimes \pi).
\]
\end{enumerate}
\end{lem}

Note that the assertions (1) and (2) in Lemma \ref{geometric} (or Proposition \ref{typeA}) 
are equivalent to each other by taking Aubert duality.
In the rest of this subsection, we suppose that we are in the situation of (1).
\par

First, we show that there is $\pi_j \in \Irr^G$ for $0 \leq j \leq ab$ such that
\begin{itemize}
\item
$\pi_0 = \pi'$ and $\pi_{ab} = \pi$; 
\item
for $0 \leq j \leq ab$, we have
$\pi_j \hookrightarrow L_j \rtimes \pi$, 
where we set
\[
L_j = D_{\rho|\cdot|^{x_j}}^{(k)} \circ \dots \circ D_{\rho|\cdot|^{x_1}}^{(k)}(S) \in \Irr^\GL;
\]
\item
$\pi_j \leq D_{\rho|\cdot|^{x_j}}^{(k)}(\pi_{j-1})$ for $1 \leq j \leq ab$.
\end{itemize}
Indeed, by induction, we assume the existence of $\pi_{j-1}$. Then 
\begin{align*}
\pi_{j-1} &\hookrightarrow L_{j-1} \rtimes \pi
\hookrightarrow (\rho|\cdot|^{x_j})^k \times L_j \rtimes \pi.
\end{align*}
By Frobenius reciprocity, we have a nonzero map 
\[
\Jac(\pi_{j-1}) \rightarrow (\rho|\cdot|^{x_j})^k \boxtimes (L_j \rtimes \pi).
\]
Hence there is an irreducible subrepresentation $\pi_j$ of 
$L_j \rtimes \pi$
such that $(\rho|\cdot|^{x_j})^k \boxtimes \pi_j$ is a subrepresentation of the image of this map.
Therefore $D_{\rho|\cdot|^{x_j}}^{(k)}(\pi_{j-1}) \geq \pi_j$.
In the particular case when $j=ab$, we find $\pi_{ab} \hookrightarrow \pi$, so we get $\pi_{ab} = \pi$.

\begin{lem}\label{geometric2}
For $0 \leq j \leq ab$, 
there is $\sigma_j \in \Irr^G$ such that $\pi_j = \soc(\tau \rtimes \sigma_j)$. 
In particular, $\sigma_{ab} = \sigma$.
Moreover, $\sigma_j \leq D_{\rho|\cdot|^{x_j}}^{(k)}(\sigma_{j-1})$ holds for $1 \leq j \leq ab$. 
\end{lem}
\begin{proof}
Note that $\tau \rtimes \sigma_j$ is SI for any $\sigma_j \in \Irr^G$ by \cite[Proposition 3.4]{At-socle}. 
Moreover, if $\pi_j = \soc(\tau \rtimes \sigma_j)$, 
and if we let $k_{B-1}, \dots, k_{-A-1} \geq k$ be the maximal integers such that 
\[
D_{\rho|\cdot|^{-A-1}}^{(k_{-A-1})} \circ \dots \circ D_{\rho|\cdot|^{B-1}}^{(k_{B-1})}(\pi) \not= 0, 
\]
then since $B-1 < 0$, it is irreducible and $\sigma_j$ is uniquely determined by 
\[
D_{\rho|\cdot|^{-A-1}}^{(k_{-A-1})} \circ \dots \circ D_{\rho|\cdot|^{B-1}}^{(k_{B-1})}(\pi)
\cong 
D_{\rho|\cdot|^{-A-1}}^{(k_{-A-1}-k)} \circ \dots \circ D_{\rho|\cdot|^{B-1}}^{(k_{B-1}-k)}(\sigma_j).
\]
In particular, we must have $\sigma_{ab} = \sigma$. 
\par

Write $\pi_j = L(\mm_j; \pi(\phi_j,\ep_j))$ as in the Langlands classification. 
First, we show that $\mm_j$ does not contain $[x,y]_\rho$ such that $x < B-1$ and $y < -A$. 
Indeed, if there were to exist such a segment in $\mm_j$, 
then we could take $C < B-1$ such that 
$D_{\rho|\cdot|^{-A-1}} \circ \dots \circ D_{\rho|\cdot|^C}(\pi_j) \not= 0$. 
However, since 
\[
\pi_j \hookrightarrow L_j \rtimes \pi
\]
and since the cuspidal support of $\Sp(\rho,a,b)$, and hence the one of $L_j$, 
does not contain $\rho|\cdot|^{\pm(A+1)}$, 
we see that there is $-A-1 \leq z < B-1$ such that $D_{\rho|\cdot|^z}(\pi) \not= 0$. 
This contradicts that $\pi = \soc(\tau \rtimes \sigma)$ and $D_{\rho|\cdot|^z}(\sigma) = 0$ for $z < B-1$.
Therefore, to show the existence of $\sigma_j \in \Irr^G$ such that $\pi_j = \soc(\tau \rtimes \sigma_j)$, 
by the Langlands classification,
it is enough to prove that $[\Jac(\pi_j)] \geq \tau \otimes \sigma_j$ for some $\sigma_j \not= 0$.
Moreover, if we were to know the existence of $\sigma_j$ such that $\pi_j = \soc(\tau \rtimes \sigma_j)$, 
then $[\Jac(\pi_j)]$ would be of the form 
\[
[\Jac(\pi_j)] \geq \tau \otimes \sigma_j + \sum_{\substack{\tau' \in \Irr^\GL \\ \tau' \not\cong \tau}} \tau' \otimes \Pi_{\tau'}.
\]
By applying a Jacquet functor to the inclusion $\pi_j \hookrightarrow L_j \times \tau \rtimes \sigma$, 
we see that $\sigma_j \leq [L_j \rtimes \sigma]$. 
In particular, since 
\[
\pi_j = \soc(\tau \rtimes \sigma_j) 
\leq D_{\rho|\cdot|^{x_j}}^{(k)}(\pi_{j-1}) 
\leq D_{\rho|\cdot|^{x_j}}^{(k)}(\tau \rtimes \sigma_{j-1}), 
\]
we would have 
\[
\sigma_j \leq \left(D_{\rho|\cdot|^{-A-1}}^{(k)} \circ \dots \circ D_{\rho|\cdot|^{-B-1}}^{(k)}\right) \circ
D_{\rho|\cdot|^{x_j}}^{(k)}(\tau \rtimes \sigma_{j-1}).
\]
By Tadi\'c's formula (Corollary \ref{TadicFormula})
together with the first assumption in Proposition \ref{typeA} (1),
it would imply that 
$\sigma_j \leq D_{\rho|\cdot|^{x_j}}^{(k)}(\sigma_{j-1})$. 
\par

Now we prove the existence of $\sigma_j$.
To do this, we prepare some notations. 
Fix $0 \leq j \leq ab$. 
Let 
\begin{itemize}
\item
$d, m \geq 0$ be such that $S \in \Irr(\GL_d(F))$ and $L_j \in \Irr(\GL_m(F))$; 
\item
$G_S$, $G_L$ and $G$ be the classical group such that 
$S \rtimes \pi \in \Rep(G_S)$, $L_j \rtimes \pi \in \Rep(G_L)$ and $L_j \times S \rtimes \pi \in \Rep(G)$; 
\item
$P = MN_P$ be the standard parabolic subgroup of $G$ with $M \cong \GL_m(F) \times G_S$;
\item
$\iota \colon \GL_m(F) \times G_S \xrightarrow{\sim} M$ be a fixed isomorphism.
\end{itemize}
There is a canonical surjection
$L_j \times S \rtimes \pi \twoheadrightarrow L_j \boxtimes (S \rtimes \pi)$
which is given by 
\[
L_j \times S \rtimes \pi \cong L_j \rtimes (S \rtimes \pi) \twoheadrightarrow L_j \boxtimes (S \rtimes \pi). 
\]
Here, $f \in L_j \rtimes (S \rtimes \pi)$ is regarded as a two variable function
\[
f \colon G \times G_S \rightarrow L_j \boxtimes S \boxtimes \pi
\]
such that 
\[
f(n \cdot \iota(a,h) \cdot g, h') = |\det(a)|^{\delta} L_j(a) f(g, h'h)
\]
for $n \in N_P, a \in \GL_m(F), h,h' \in G_S$ and $g \in G$, 
where $\delta \in \R$ is independent of $f$.
Then the isomorphism $L_j \rtimes (S \rtimes \pi) \xrightarrow{\sim} L_j \times S \rtimes \pi$
is given by $f(g,h) \mapsto f(g,\1)$, 
whereas, the surjection $L_j \rtimes (S \rtimes \pi) \twoheadrightarrow L_j \boxtimes (S \rtimes \pi)$
is given by $f(g,h) \mapsto f(\1,h)$.
In particular, the surjection 
$L_j \times S \rtimes \pi \twoheadrightarrow L_j \boxtimes (S \rtimes \pi)$
is described by $f(g, \1) \mapsto f(\iota(\1,h), \1)$.
This surjection factors through the $M$-equivariant map
\[
\Jac_{P}(L_j \times S \rtimes \pi) \twoheadrightarrow L_j \boxtimes (S \rtimes \pi),
\]
which we denote by $\Res_{G_S}$.
Similarly, we have an evaluation map 
$\ev \colon \Jac(L_j \rtimes \pi) \twoheadrightarrow L_j \boxtimes \pi$ at $\1$.
This gives a surjection 
\[
(\1_{\GL_0(F)} \boxtimes S) \rtimes \Jac(L_j \rtimes \pi)
\twoheadrightarrow L_j \boxtimes (S \rtimes \pi),
\]
which is also denoted by $\ev$. 
By applying the contragredient and the MVW functors to the injection $\pi' \hookrightarrow S \rtimes \pi$, 
we obtain, by \cite[Lemma 2.2]{AG}, a surjection  $S \rtimes \pi \twoheadrightarrow \pi'$. 
\par

Note that $S \times L_j$ is irreducible by \cite[Proposition 6.6]{LM},
and hence $S \times L_j \rtimes \pi \cong L_j \times S \rtimes \pi$.
We have an inclusion map $S \boxtimes \pi_j \hookrightarrow S \boxtimes (L_j \rtimes \pi)$ 
of representations of $\GL_d(F) \times G_L$, which is regarded as a Levi subgroup of $G$.
It gives an inclusion
\[
\Jac_P(S \rtimes \pi_j) \hookrightarrow \Jac_P(S \rtimes (L_j \rtimes \pi)).
\]
By the functoriality of the Geometric Lemma (Theorem \ref{GeometricLemma}), 
we have subspaces $F(S \boxtimes \pi_j) \subset S \rtimes \pi_j$ 
and $F(S \boxtimes (L_j \rtimes \pi)) \subset S \rtimes (L_j \rtimes \pi)$ 
such that there is a commutative diagram
\[
\xymatrix{
\Jac_P(F(S \boxtimes \pi_j)) \;\ar@{->>}[d] \ar@{^(->}[rr] 
&& \Jac_P(F(S \boxtimes (L_j \rtimes \pi))) \ar@{->>}[d] \\
(\1_{\GL_0(F)} \boxtimes S) \rtimes \Jac(\pi_j) \;\ar@{^(->}[rr]
&&(\1_{\GL_0(F)} \boxtimes S) \rtimes \Jac(L_j \rtimes \pi).
}
\]
Incorporating the maps obtained in the previous paragraph into the picture, 
we obtain the following  diagram:
\[
\xymatrix{
\Jac_P(S \rtimes \pi_j) \; \ar@{^(->}[r] 
& \Jac_P(S \times L_j \rtimes \pi) \ar@{=}[r]
& \Jac_P(L_j \times S \rtimes \pi) \ar@{->>}[d]^{\Res_{G_S}}\\
\Jac_P(F(S \boxtimes \pi_j)) \; \ar@{}[u]|{\bigcup} \ar@{->>}[d] \ar@{^(->}[r] 
& \Jac_P(F(S \boxtimes (L_j \rtimes \pi))) \ar@{}[u]|{\bigcup} \ar@{->>}[d] 
& L_j \boxtimes (S \rtimes \pi) \ar@{->>}[d]
\\
(\1_{\GL_0(F)} \boxtimes S) \rtimes \Jac(\pi_j) \;\ar@{^(->}[r]
&(\1_{\GL_0(F)} \boxtimes S) \rtimes \Jac(L_j \rtimes \pi) \ar@{->>}[ur]_{\ev}
& L_j \boxtimes \pi'.
}
\]
We claim that this diagram is commutative. 
In fact, the isomorphism $S \times L_j \rtimes \pi \xrightarrow{\sim} L_j \times S \rtimes \pi$
is given by the meromorphic continuation of the Jacquet integral 
\[
Jf(g) = \int_{\tl{N}_P} f(w^{-1}ug) du,
\]
where $w \in G$ is a representative of a certain Weyl element, and $\tl{N}_P$ is a quotient of $N_P$. 
In particular, the composition with $\Res_{G_S}$ is induced by 
$f \mapsto Jf(\iota(\1,h))$.
On the other hand, the subspace $F(S \boxtimes (L_j \rtimes \pi))$ of $S \times L_j \rtimes \pi$
is taken as a subspace consisting of $f$
such that $Jf(\iota(a,h))$ converges absolutely for $(a,h) \in \GL_m(F) \times G_S$, 
and this is the image of $f$ under the surjection 
$F(S \boxtimes (L_j \rtimes \pi)) 
\twoheadrightarrow (\1_{\GL_0(F)} \boxtimes S) \rtimes \Jac(L_j \rtimes \pi)$.
Hence the composition with $\ev$ is given by $f \mapsto Jf(\iota(\1,h))$.
This shows the commutativity of the right square. 
\par

Note that the composition 
\[
(\1_{\GL_0(F)} \boxtimes S) \rtimes \Jac(\pi_j) 
\hookrightarrow (\1_{\GL_0(F)} \boxtimes S) \rtimes \Jac(L_j \rtimes \pi) 
\xrightarrow{\ev} L_j \boxtimes (S \rtimes \pi)
\]
is surjective since it is induced from a surjection $\Jac(\pi_j) \twoheadrightarrow L_j \boxtimes \pi$. 
Hence the above commutative diagram shows that
we have a surjection $\Jac(S \rtimes \pi_j) \twoheadrightarrow L_j \boxtimes \pi'$, 
which implies, by Frobenius reciprocity, a nonzero map
\[
S \rtimes \pi_j \rightarrow L_j \rtimes \pi'.
\]
Since $L_j \times \soc(\tau \times S)$ is irreducible by \cite[Proposition 6.6]{LM}, 
we have
\begin{align*}
S \rtimes \pi_j 
&\rightarrow L_j \rtimes \pi'
\\&\hookrightarrow L_j \times \soc(\tau \times S) \rtimes \sigma
\\&\cong \soc(\tau \times S) \times L_j \rtimes \sigma
\\&\hookrightarrow \tau \times S \times L_j \rtimes \sigma.
\end{align*}
This implies $[\Jac(S \rtimes \pi_j)] \geq \tau \otimes S \otimes L_j \otimes \sigma \not= 0$. 
Since $\rho|\cdot|^{-A-1} \not\in \supp(S)$, 
by Tadi\'c's formula (Corollary \ref{TadicFormula}), 
there is an irreducible representation $\tau'' \otimes \sigma'_j$ with 
$[\Jac(\pi_j)] \geq \tau'' \otimes \sigma_j'$ such that
\[
\{\underbrace{\rho|\cdot|^{-A-1}, \dots, \rho|\cdot|^{-A-1}}_k\}
\subset \supp(\tau'') \subset \supp(\tau)
\]
as multi-sets, where $\supp(\tau)$ is the cuspidal support of $\tau$. 
By the condition on $\mm_j$ we have proven earlier, 
we must have $\tau''=\tau$, as desired. 
This completes the proof of Lemma \ref{geometric2}.
\end{proof}

Now we can prove Lemma \ref{geometric}. 
\begin{proof}[Proof of Lemma \ref{geometric}]
As said before, it is enough to show only (1). 
So we suppose that we are in the situation of (1).  
\par

For $\sigma_j \in \Irr^G$ defined in Lemma \ref{geometric2}, 
we claim that 
\[
\left(D_{\rho|\cdot|^{x_j}}^{(k)} \circ \dots \circ D_{\rho|\cdot|^{x_1}}^{(k)}\right)
(\soc(\soc(S' \rtimes \tau) \rtimes \sigma_0))
\geq \soc(\soc(S' \times \tau) \rtimes \sigma_j)
\]
for $0 \leq j \leq ab$.
This claim for $j=ab$ yields that 
\[
D_S(\soc(\soc(S' \rtimes \tau) \rtimes \sigma_0))
\geq \soc(\soc(S' \times \tau) \rtimes \sigma_{ab}) = \soc(S' \rtimes \pi).
\]
Since
\begin{align*}
\soc(\soc(S' \times \tau) \rtimes \sigma_0) 
&= \soc(S' \rtimes \soc(\tau \rtimes \sigma_0))
\\&= \soc(S' \rtimes \pi')
\leq \soc(\soc(S' \times \tau \times S) \rtimes \sigma),
\end{align*}
this implies Lemma \ref{geometric}. 
\par

The claim for $j=0$ is trivial. 
Fix $1 \leq j \leq ab$, and suppose that the claim holds for $j-1$.
Then 
\[
\left(D_{\rho|\cdot|^{x_j}}^{(k)} \circ \dots \circ D_{\rho|\cdot|^{x_1}}^{(k)}\right)
(\soc(\soc(S' \rtimes \tau) \rtimes \sigma_0))
\geq D_{\rho|\cdot|^{x_j}}^{(k)}(\soc(\soc(S' \times \tau) \rtimes \sigma_{j-1})).
\]
\par

Recall that $-A \leq x_j \leq A$. 
First, we assume $x_j \not= B$
so that $\soc(S' \rtimes \tau) \times (\rho|\cdot|^{x_j})^k$ is irreducible. 
Since 
\[
\left[\Jac(\soc(\soc(S' \times \tau) \rtimes \sigma_{j-1})\right]
\geq \soc(S' \times \tau) \otimes (\rho|\cdot|^{x_j})^k \otimes \sigma_{j}, 
\]
there is $\tau' \in \Irr^\GL$ such that 
$[\Jac(\soc(\soc(S' \times \tau) \rtimes \sigma_{j-1})] \geq \tau' \otimes \sigma_j$
and $[\Jac(\tau')] \geq \soc(S' \times \tau) \otimes (\rho|\cdot|^{x_j})^k$.
\par

We show that $\tau' = \soc(S' \times \tau) \times (\rho|\cdot|^{x_j})^k$.
We write $\tau' = L(\mm)$ and 
$\soc(S' \times \tau) = L(\Delta_\rho[a_1, b_1]^{k_1}, \dots, \Delta_\rho[a_r,b_r]^{k_r})$
with $b_1 < \dots < b_{r-1} < b_r = -A-1$ and $k_r = k$.
By looking at the cuspidal support, we can find $a_{u,v} \geq b_u$ 
for $1 \leq u \leq r$ and $1 \leq v \leq k_u$ such that 
\[
\sum_{u=1}^r \sum_{v=1}^{k_u} [a_{u,v}, b_u]_\rho \leq \mm.
\]
Suppose that $a_{u,v} \not= a_u$ for some $(u,v)$. 
Take the minimum $u$ with this condition. 
Hence $a_{i,v} = a_i$ for any $i < u$ and $1 \leq v \leq k_i$. 
If $a_{u,v} > a_u$, then $[\Jac(\tau')] \not\geq \soc(S' \times \tau) \otimes (\rho|\cdot|^{x_j})^k$. 
(See cf., \cite{KL} for $\Jac(\soc(S' \times \tau))$.)
This is a contradiction.
It $a_{u,v} < a_u$, then 
\[
\left(D_{\rho|\cdot|^{b_u}} \circ \dots \circ D_{\rho|\cdot|^{a_{u,v}}}\right)
\circ \left( 
\mathop{\circ}_{1 \leq i < u} D_{\rho|\cdot|^{b_i}}^{(k_i)} \circ \dots \circ D_{\rho|\cdot|^{a_i}}^{(k_i)} 
\right) 
(\soc(\soc(S' \times \tau) \rtimes \sigma_{j-1})) \not= 0.
\]
Since $a_u < 0$ 
and $D_{\rho|\cdot|^{b_i}} \circ \dots \circ D_{\rho|\cdot|^{a_i}}(\sigma_{j-1}) = 0$ for $i < u$, 
it implies that 
$D_{\rho|\cdot|^{b_u}} \circ \dots \circ D_{\rho|\cdot|^{a_{u,v}}}(\sigma_{j-1}) \not= 0$.
However, since $\sigma_{j-1} \leq [L_{j-1} \rtimes \sigma]$, 
we can find $b_u \leq c \leq a_{u,v}$ such that $D_{\rho|\cdot|^c}(\sigma) \not= 0$.
This contradicts that $D_{\rho|\cdot|^x}(\sigma) = 0$ for any $x < B-1 = \max\{a_1, \dots, a_r\}$.
Therefore, $a_{u,v} = a_u$ for any $(u,v)$, and hence 
\[
\sum_{u=1}^r k_u [a_u, b_u]_\rho \leq \mm.
\]
Comparing the cuspidal supports, we see that the difference is $k [x_j,x_j]_\rho$.
Hence 
\[
\tau' = \soc(S' \times \tau) \times (\rho|\cdot|^{x_j})^k \cong (\rho|\cdot|^{x_j})^k \times \soc(S' \rtimes \tau).
\]
In particular, we have
\[
\left[\Jac(\soc(\soc(S' \times \tau) \rtimes \sigma_{j-1})\right]
\geq (\rho|\cdot|^{x_j})^k \otimes \soc(S' \times \tau) \otimes \sigma_{j}.
\]
It implies that there is $\pi_j' \in \Irr^G$ such that 
$D_{\rho|\cdot|^{x_j}}^{(k)}(\soc(\soc(S' \times \tau) \rtimes \sigma_{j-1})) \geq \pi_j'$
and $[\Jac(\pi_j')] \geq \soc(S' \times \tau) \otimes \sigma_j$.
\par

We show that $\pi_j' = \soc(\soc(S' \times \tau) \rtimes \sigma_j)$. 
We write $\pi_j' = L(\mm; \pi(\phi, \ep))$ with $\pi(\phi,\ep)$ tempered.
Since $[\Jac(\pi_j')] \geq \soc(S' \times \tau) \otimes \sigma_j$ and since $a_u < 0$, 
we see that 
\[
\sum_{u=1}^r\sum_{v=1}^{k_u} [a_{u,v}, b_u]_\rho \leq \mm
\]
for some $b_u \leq a_{u,v} \leq a_u$.
Suppose that $a_{u,v} < a_u$ for some $(u,v)$, and we take the minimum such $u$.
Then 
\[
\left(D_{\rho|\cdot|^{b_u}} \circ \dots \circ D_{\rho|\cdot|^{a_{u,v}}} \right)
\circ \left( 
\mathop{\circ}_{1 \leq i < u} D_{\rho|\cdot|^{b_i}}^{(k_i)} \circ \dots \circ D_{\rho|\cdot|^{a_i}}^{(k_i)} 
\right) 
\circ 
D_{\rho|\cdot|^{x_j}}^{(k)}(\soc(\soc(S' \times \tau) \rtimes \sigma_{j-1})) \not= 0.
\]
This implies $D_{\rho|\cdot|^c}(\sigma) \not= 0$ for some $b_u \leq c < a_u \leq B-1$, 
or $D_{\rho|\cdot|^{b_u}} \circ \dots \circ D_{\rho|\cdot|^{B-1}}(\sigma) \not= 0$. 
This contradicts the first assumption in Proposition \ref{typeA} (1). 
Therefore, $a_{u,v} = a_u$ for any $(u,v)$, and hence 
\[
\sum_{u=1}^r k_u [a_u, b_u]_\rho \leq \mm.
\]
Moreover, by the same assumption in Proposition \ref{typeA} (1), 
we see that $D_{\rho|\cdot|^{-A-1}} \circ \dots \circ D_{\rho|\cdot|^x}(\sigma_{j-1}) = 0$ 
for any $-A-1 \leq x \leq B-2$.
This implies that
$L(\mm) = \soc(\soc(S' \times \tau) \times L(\mm'))$ for some $\mm' \leq \mm$.
Hence we can write $\pi_j' = \soc(\soc(S' \times \tau) \rtimes \sigma_j')$ for some $\sigma_j' \in \Irr^G$. 
Since $\pi'_j \leq D_{\rho|\cdot|^{x_j}}^{(k)}(\soc(\soc(S' \times \tau) \rtimes \sigma_{j-1}))$, 
we see that $\sigma_j'$ is uniquely determined by the inequality 
\[
[\Jac(\pi_j')] \geq \Delta_{\rho}[a_1,b_1]^{k_1} \otimes \dots \otimes \Delta_\rho[a_r,b_r]^{k_r}\otimes \sigma_j'.
\]
Hence we have $\sigma_j' = \sigma_j$, 
and we conclude that $\pi_j' = \soc(\soc(S' \times \tau) \rtimes \sigma_j)$. 
Therefore
\[
D_{\rho|\cdot|^{x_j}}^{(k)}(\soc(\soc(S' \times \tau) \rtimes \sigma_{j-1})) 
\geq \soc(\soc(S' \times \tau) \rtimes \sigma_j),
\]
as desired. 
We obtain the claim for $j$ with $x_j \not= B$.
\par

Next, we assume that $x_j = B$.  
By the proof of \cite[Proposition 3.4]{At-socle}, we have
\[
\soc(\soc(S' \times \tau) \rtimes \sigma_{j-1}) 
= \soc(S' \rtimes \soc(\tau \rtimes \sigma_{j-1}))
= \soc(S' \rtimes \pi_{j-1}). 
\]
Hence by the same argument starting with 
\[
\left[\Jac(\soc(S' \rtimes \pi_{j-1}))\right] \geq S' \otimes (\rho|\cdot|^{x_j})^k \otimes \pi_j,
\]
we have
\[
D_{\rho|\cdot|^{x_j}}^{(k)}(\soc(S' \rtimes \pi_{j-1})) \geq \soc(S' \rtimes \pi_j).
\]
Since $\soc(S' \rtimes \pi_j) = \soc(S' \rtimes \soc(\tau \rtimes \sigma_j)) 
= \soc(\soc(S' \times \tau) \rtimes \sigma_j)$, 
we obtain the claim for $j$ with $x_j = B$.
This completes the proof of Lemma \ref{geometric}.
\end{proof}

%\subsection{M{\oe}glin's construction}
\subsection{M{\oe}glin's construction}\label{sec.Moe}
To show Proposition \ref{typeA}, we review M{\oe}glin's explicit construction of $A$-packets.
It was refined by the first author in \cite{At-const}, in which
the following notion was introduced.

\begin{defi}
\begin{enumerate}
\item
An \emph{extended segment} is a triple $([A,B]_\rho, l, \eta)$, 
where 
\begin{itemize}
\item
$[A,B]_\rho$ is a segment with $A,B \in (1/2)\Z$ (so that $A \geq B$ and $A-B \in \Z$); 
\item
$l \in \Z$ such that $0 \leq l \leq \half{b}$ with $b = A-B+1$; 
\item
$\eta \in \{\pm1\}$. 
\end{itemize}

\item
Two extended segments $([A,B]_\rho, l, \eta)$ and $([A',B']_{\rho'}, l', \eta')$
are said to be \emph{equivalent} if 
\begin{itemize}
\item
$[A,B]_\rho = [A',B']_{\rho'}$; 
\item
$l=l'$; and 
\item
$\eta = \eta'$ if $l=l' < \half{b}$. 
\end{itemize}
Two multi-sets of extended segments 
$\{([A_i,B_i]_{\rho_i}, l_i, \eta_i)\}_{i \in I}$ and $\{([A'_i,B'_i]_{\rho'_i}, l'_i, \eta'_i)\}_{i \in I}$
with the same index set $I$ are said to be \emph{equivalent} if 
$([A_i,B_i]_{\rho_i}, l_i, \eta_i)$ and $([A'_i,B'_i]_{\rho'_i}, l'_i, \eta'_i)$ are equivalent
for all $i \in I$. 

\item
An \emph{extended multi-segment} for $G$
is an equivalence class of multi-sets of extended segments
\[
\EE = \{([A_i,B_i]_{\rho_i}, l_i, \eta_i)\}_{i \in (I,>)}
\]
such that 
\begin{itemize}
\item
$\rho_i \in \Cusp^\GL$ with $\rho_i \cong \rho_i^\vee$; 
\item
$A_i+B_i \geq 0$; 
\item
$\psi_\EE = \oplus_{i \in I} \rho_i \boxtimes S_{a_i} \boxtimes S_{b_i}$ is an $A$-parameter for $G$ of good parity, 
where $a_i = A_i+B_i+1$ and $b_i = A_i-B_i+1$; 
\item
the sign condition 
\[
\prod_{i \in I} (-1)^{[\half{b_i}]+l_i} \eta_i^{b_i} = 1
\]
holds; 
\item
$>$ is a total order on $I$, 
which is \emph{admissible} in the sense that 
\[
\rho_i \cong \rho_j,\; A_i < A_j,\; B_i < B_j \implies i < j
\]
and 
\[
\rho_i \cong \rho_j,\; B_i < B_j < 0 \implies i < j.
\]
\end{itemize}
\end{enumerate}
\end{defi}

Let $\EE = \{([A_i,B_i]_{\rho_i}, l_i, \eta_i)\}_{i \in (I,>)}$ be an extended multi-segment for $G$. 
When $i < j$, we write $([A_i,B_i]_{\rho_i}, l_i, \eta_i) < ([A_j,B_j]_{\rho_j}, l_j, \eta_j)$.
For simplicity, we write $I = \{1, \dots, m\}$ with the usual order. 
We say that $\EE$ has a \emph{non-negative discrete diagonal restriction (DDR)} 
if 
\[
0 \leq B_1 \leq A_1 < B_2 \leq A_2 < \dots < B_m \leq A_m. 
\]
In this case, we define $\pi(\EE) \in \Irr(G)$ by 
\[
\pi(\EE) = \soc\left(
\bigtimes_{i = 1}^m \begin{pmatrix}
B_i & \ldots & B_i+l_i-1 \\
\vdots & \ddots & \vdots \\
-A_i & \ldots & -(A_i-l_i+1)
\end{pmatrix}_{\rho_i}
\rtimes \pi(\phi, \ep)
\right)
\]
with 
\[
\phi = \bigoplus_{i=1}^m \bigoplus_{j=0}^{A_i-B_i-2l_i} \rho_i \boxtimes S_{2(B_i+l_i+j)+1}
\]
and $\ep(\rho \boxtimes S_{2(B_i+l_i+j)+1}) = (-1)^{j}\eta_i$ for $1 \leq i \leq m$ 
and $0 \leq j \leq A_i-B_i-2l_i$.
\par

In the general case, given an extended multi-segment 
$\EE = \{([A_i,B_i]_{\rho_i}, l_i, \eta_i)\}_{i \in (I,>)}$, 
to define the representation $\pi(\EE)$, we proceed as follows. 
Take a sequence of non-negative integers  $\{t_i\}_{i \in (I,>)}$ such that 
$\EE_{\gg} = \{([A_i+t_i,B_i+t_i]_{\rho_i}, l_i, \eta_i)\}_{i \in (I,>)}$
has a non-negative DDR, and define 
\[
\pi(\EE) = \mathop{\circ}_{1 \leq i \leq m}\left(
D_{\rho_i|\cdot|^{B_i+1}, \dots, \rho_i|\cdot|^{A_i+1}} 
\circ \dots \circ D_{\rho_i|\cdot|^{B_i+t_i}, \dots, \rho_i|\cdot|^{A_i+t_i}}
\right)(\pi(\EE_\gg)). 
\]
This definition does not depend on the choice of $\{t_i\}_{i \in (I,>)}$.
Then $\pi(\EE)$ is irreducible or zero. 

\begin{thm}[{\cite[Theorem 1.2]{At-const}}]
For $\psi \in \Psi(G)$ of good parity, we have
\[
\Pi_\psi = \{\pi(\EE) \;|\; \psi_\EE \cong \psi\} \setminus \{0\}.
\]
\end{thm}

%\subsection{Proof of Proposition \ref{typeA}}
\subsection{Proof of Proposition \ref{typeA}}\label{sec.proof_inductive}
In this subsection, we prove Proposition \ref{typeA}.
Note that the assertions (1) and (2) are equivalent to each other by taking Aubert duality.
Hence we only prove (2). 
\par

Recall that $\sigma$ is a representation of Arthur type such that 
$D_{\rho|\cdot|^x}(\sigma) = 0$ for $x > B+1$, 
and that $D_{\rho|\cdot|^{A+1}} \circ \dots \circ D_{\rho|\cdot|^{B+1}}(\sigma) = 0$. 
By \cite[Theorem 4.1]{At-Is_Arthur}, 
there is an extended multi-segment $\EE$ with $\sigma = \pi(\EE)$
such that if $([A',B']_{\rho}, *, *) \in \EE$, then $B' \leq B$, or $B' = B+1$ and $A' \leq A$.
\par

We take integers $t_1,\dots,t_k \gg 0$, which define $S'$.
We will give conditions for these integers later. 
By \cite[Theorem 1.1]{At-socle}, in the notation of Lemma \ref{geometric} (2), 
any irreducible summand $\pi''$ in $\soc(\soc(S' \times \tau \times S) \rtimes \sigma)$ is of Arthur type, 
and has an extended multi-segment of the form
\[
\EE'' = \EE \cup \bigcup_{i=1}^k\left\{([A,B]_\rho,l,\eta_i), ([A+t_i,B+t_i]_\rho,l,\eta'_i)\right\}, 
\]
where the order $<$ on $\EE''$ is given by
\begin{align*}
([A',B']_{\rho'}, l', \eta') &< ([A,B]_\rho,l,\eta_1) < \dots < ([A,B]_\rho,l,\eta_k)
\\&
< ([A+t_1,B+t_1]_\rho,l,\eta'_1) < \dots < ([A+t_k,B+t_k]_\rho,l,\eta'_k)
\end{align*}
for any $([A',B']_\rho, l', \eta') \in \EE$. 
Notice that this is an admissible order 
since if $\rho' \cong \rho$ and $B' = B+1$, then $[A',B']_{\rho'} \subset [A,B]_{\rho}$.
\par

We write 
\[
\EE \cup \{([A,B]_\rho,l,\eta_i)\;|\; i = 1,\dots,k\} 
= \{([A'_j,B'_j]_{\rho'}, l'_j,\eta'_j) \;|\; j = 1, \dots, k'\}, 
\]
which is independent of $t_1,\dots,t_k$. 
Now we choose integers $t_1,\dots,t_k \gg0$ so that 
there are integers $t_1',\dots,t'_{k'} \geq 0$ such that
\[
\EE''_{\gg} = 
\left\{([A'_j+t'_j,B'_j+t'_j]_{\rho'}, l'_j,\eta'_j) \;\middle|\; j = 1,\dots, k'\right\}
\cup
\left\{([A+t_i,B+t_i]_\rho,l,\eta'_i) \;\middle|\; i=1,\dots,k\right\}
\]
has a non-negative DDR, i.e., 
\begin{align*}
0 \leq B'_1+t'_1 &\leq A'_1+t'_1 < \dots < B'_{k'}+t'_{k'} \leq A'_{k'}+t'_{k'}
\\&
< B+t_1 \leq A+t_1 < \dots < B+t_k \leq A+t_k.
\end{align*}
In particular, $B+t_1 > A$.
Then by construction, 
one can write 
\[
\pi(\EE'') = L(\mm_1+\mm_0; \pi(\phi_1 \oplus \phi_0, \ep)),
\] 
where 
\begin{itemize}
\item
$\mm_1$ is the multi-segment given by 
\[
\mm_1 = \sum_{i=1}^k \sum_{j=0}^{l-1} [B+t_i+j, -A-t_i+j]_\rho;
\]
\item
$\phi_1$ is the representation of $W_F \times \SL_2(\C)$ given by
\[
\phi_1 = \bigoplus_{i=1}^k \bigoplus_{j=0}^{A-B-2l} \rho \boxtimes S_{2(B+t_i+l+j)+1}; 
\]
\item
$\ep$ satisfies that 
\[
\ep(\rho \boxtimes S_{2(B+t_i+l+j)+1}) = (-1)^{j}\eta_i'
\]
for $1 \leq i \leq k$ and $0 \leq j \leq A-B-2l$;
\item
$\mm_0$ is a multi-segment such that 
if $[x,-y]_\rho \in \mm_0$, then $x < y < B+t_1$; 
\item
$\phi_0$ is a tempered representation of $W_F \times \SL_2(\C)$ such that 
if $\rho \boxtimes S_{2d+1} \subset \phi_0$, then $d < B+t_1$.
\end{itemize}
\par

Now, by Lemma \ref{geometric}, 
we can take $\pi(\EE'') \hookrightarrow \soc(\soc(S' \times \tau \times S) \rtimes \sigma)$ such that 
\[
D_S(\pi(\EE'')) \geq \soc(S' \rtimes \pi).
\] 

\begin{lem}\label{remaining}
We have 
$\soc(S' \rtimes \pi) = L(\mm_1+\mm_0'; \pi(\phi_1 \oplus \phi_0', \ep'))$
for some $\mm_0', \phi_0', \ep'$ with $\ep'|_{\phi_1} = \ep|_{\phi_1}$. 
In other words, 
$\mm_1$ and $(\phi_1, \ep|_{\phi_1})$ remain in the Langlands data of $\soc(S' \rtimes \pi)$.
\end{lem}
\begin{proof}
We write $\soc(S' \rtimes \pi) = L(\mm'; \pi(\phi', \ep'))$.
For $1 \leq i \leq k$ and $0 \leq j \leq l-1$, 
assuming that $\mm'$ contains 
\[
\sum_{i'=i+1}^{k} \sum_{j'=0}^{l-1} [B+t_{i'}+j', -A-t_{i'}+j']_\rho 
+\sum_{j'=0}^{j-1} [B+t_i+j', -A-t_i+j']_\rho,
\] 
we will show $[B+t_i+j, -A-t_i+j]_\rho \in \mm'$. 
\par

Suppose first that for any $[x,y]_\rho \in \mm'$, we have $y \not= -A-t_i+j$.
Then by looking at the cuspidal support, 
we see that $\phi' \supset \rho \boxtimes S_{2(A+t_i-j)+1}$.
Note that $D_{\rho|\cdot|^x}(\soc(S' \rtimes \pi)) = 0$ for $B+t_i < x \leq A+t_i$ 
since the same condition holds for $\pi(\EE'')$ and $t_i \gg0$. 
It implies that
\[
\phi' \supset \rho \boxtimes \left(S_{2(B+t_i+j)+1} \oplus \dots \oplus S_{2(A+t_i-j)+1}\right).
\]
Then for any sequence $x_1, \dots, x_m$ with $x_m = -A-t_i+j$, 
we have 
\[
D_{\rho|\cdot|^{x_m}} \circ \dots \circ D_{\rho|\cdot|^{x_1}}(L(\mm'; \pi(\phi', \ep'))) = 0.
\] 
In particular, we have $\Jac(L(\mm'; \pi(\phi', \ep'))) \not\geq S' \otimes \pi$, which is a contradiction. 
Hence there is $[x, y]_\rho \in \mm'$ such that $y = -A-t_i+j$.
\par

Let $[x, -A-t_i+j]_\rho \in \mm'$.
Note that $x < A+t_i-j$.
If $x > B+t_i+j$, then since $D_{\rho|\cdot|^x}(\pi(\EE'')) = 0$ 
and hence $D_{\rho|\cdot|^x}(L(\mm'; \pi(\phi', \ep'))) = 0$, 
we must have $[x-1,y]_\rho \in \mm'$ such that $y < -A-t_i+j$.
However, by looking at the cuspidal support, we cannot take such $y$. 
If $x < B+t_i+j$, then 
\[
D_{\rho|\cdot|^x, \dots, \rho|\cdot|^{-A-t_i+j}}
\circ 
\left(
\mathop{\circ}_{0 \leq j' <j}
D_{\rho|\cdot|^{B+t_i+j'}, \dots, \rho|\cdot|^{-A-t_i+j'}}
\right)
(L(\mm'; \pi(\phi', \ep'))) \not= 0,
\] 
and hence 
\[
D_{\rho|\cdot|^x, \dots, \rho|\cdot|^{-A-t_i+j}} 
\circ 
\left(
\mathop{\circ}_{0 \leq j' <j}
D_{\rho|\cdot|^{B+t_i+j'}, \dots, \rho|\cdot|^{-A-t_i+j'}}
\right)
\circ D_S(\pi(\EE'')) \not= 0.
\]
Since $t_i \gg0$, this is impossible. 
Therefore, we have $x = B+t_i+j$, and we conclude that $[B+t_i+j, -A-t_i+j]_\rho \in \mm'$.
\par

By induction, we see that $\soc(S' \rtimes \pi) = \soc(L(\mm_1) \rtimes \pi_0)$
for some irreducible representation $\pi_0$. 
Since 
\begin{align*}
\soc(S' \rtimes \pi) &\leq D_S(\pi(\EE'')) 
\\&\leq D_S\left(L(\mm_1) \rtimes L(\mm_0; \pi(\phi_1 \oplus \phi_0, \ep))\right)
\\&= L(\mm_1) \rtimes D_S\left(L(\mm_0; \pi(\phi_1 \oplus \phi_0, \ep))\right),
\end{align*}
we see that $\pi_0 \leq D_S(L(\mm_0; \pi(\phi_1 \oplus \phi_0, \ep)))$.
Since $t_1, \dots, t_k \gg0$, we see that $\pi_0 = L(\mm_0'; \pi(\phi_1 \oplus \phi_0', \ep'))$ 
for some $\mm_0'$ and $(\phi_0', \ep')$. 
Hence $\soc(S' \rtimes \pi) = L(\mm_1+\mm_0'; \pi(\phi_1 \oplus \phi_0', \ep'))$, as desired. 
\end{proof}

Recall that $\sigma$ is of Arthur type. 
\begin{prop}\label{changingE}
Continue in the above setting.
There is an extended multi-segment $\EE_0$ such that 
\begin{itemize}
\item
$\pi(\EE_0) = \sigma$; 
\item
$\EE_0$ contains $\{([A,B]_\rho,l,\eta'_i) \;|\; i=1,\dots,k\}$; 
\item
if $([A',B'],l',\eta') \in \EE_0$ is not in $\{([A,B]_\rho,l,\eta'_i) \;|\; i=1,\dots,k\}$, 
then $([A',B'],l',\eta') < ([A,B]_\rho,l,\eta'_i)$. 
\end{itemize}
\end{prop}

Admitting this proposition, 
by \cite[Algorithm 3.3]{At-Is_Arthur}, we conclude that $\pi = \soc(\tau \rtimes \sigma)$ is of Arthur type. 
In fact, $\pi$ has an extended multi-segment $\EE_1$ 
obtained from $\EE_0$ by replacing 
$([A,B]_\rho,l,\eta'_i)$ with $([A+1,B+1]_\rho,l,\eta'_i)$ for $i=1, \dots,k$.
This will complete the proof of Proposition \ref{typeA}.

\begin{proof}[Proof of Proposition \ref{changingE}]
First of all, by \cite[Theorem 4.1]{At-Is_Arthur}, 
there is an extended multi-segment $\EE$ with $\sigma = \pi(\EE)$
such that if $([A',B']_{\rho}, *, *) \in \EE$, then $B' \leq B$, or $B' = B+1$ and $A' \leq A$.
We choose such an $\EE$ so that 
the number 
\[
X = \sum_{\substack{([A',B']_\rho, l', \eta') \in \EE \\ B'=B}} A'
\]
is maximal. 
By changing the admissible order if necessary, 
we may assume that we can write
\[
\EE = \EE_1 \cup \{([A_i,B]_\rho, l_i, \eta_i)\}_{i=1,\dots,m} \cup \{([A_i',B+1]_\rho, l_i', \eta_i')\}_{i=1,\dots,m'}, 
\]
where 
\begin{itemize}
\item
$\{([A_i',B+1]_\rho, l_i', \eta_i')\}_{i=1,\dots,m'}$ 
is the multi-set contained $\EE$ consisting of $([A',B']_\rho, *, *)$ 
such that $B'=B+1$, hence $A_i' \leq A$; 
\item
$\{([A_i,B]_\rho, l_i, \eta_i)\}_{i=1,\dots,m}$ 
is the multi-set contained $\EE$ consisting of $([A',B']_\rho, *, *)$ 
such that $B' = B$ and $A' \geq A$; 
\item
$\EE_1$ is the multi-set contained $\EE$ consisting of other $([A',B']_{\rho'}, *, *)$; 
\item
the admissible order is so that 
\[
([A',B']_{\rho'}, *, *) < ([A_i,B]_\rho, l_i, \eta_i) < ([A_i',B+1]_\rho, l_i', \eta_i')
\]
for $([A',B']_{\rho'}, *, *) \in \EE_1$, 
and the canonical orders on $\{1, \dots, m\}$ and $\{1, \dots, m'\}$.
\end{itemize}
Note that $[A_i,B]_\rho \supset [A_j',B+1]_\rho$ for any $1 \leq i \leq m$ and $1 \leq j \leq m'$.
Hence by changing the admissible order and updating the notations, 
we can rewrite 
\[
\EE = \EE'_1 \cup \{([A,B_i']_\rho, l'_i, \eta'_i)\}_{i=1,\dots,m'} 
\cup \{([A_i,B]_\rho, l_i, \eta_i)\}_{i=1,\dots,m}
\]
such that 
\begin{itemize}
\item
$\{([A_i,B]_\rho, l_i, \eta_i)\}_{i=1,\dots,m}$ 
is the multi-set contained $\EE$ consisting of $([A',B']_\rho, *, *)$ 
such that $B' = B$ and $A' \geq A$; 
\item
$\{([A,B_i']_\rho, l'_i, \eta'_i)\}_{i=1,\dots,m'}$ 
is the multi-set contained $\EE$ consisting of $([A',B']_\rho, *, *)$ 
such that $B' < B$ and $A' = A$; 
\item
$\EE'_1$ is the multi-set contained $\EE$ consisting of other $([A',B']_{\rho'}, *, *)$; 
\item
the admissible order is so that 
\[
([A',B']_{\rho'}, *, *) < ([A,B_i']_\rho, l_i', \eta_i') < ([A_i,B]_\rho, l_i, \eta_i) 
\]
for $([A',B']_{\rho'}, *, *) \in \EE'_1$, 
and the canonical orders on $\{1, \dots, m\}$ and $\{1, \dots, m'\}$.
\end{itemize}
Note that $m$ and $m'$ can be equal to zero. 
\par

Let $m_0 = |\{i \in \{1, \dots, m\} \;|\; A_i=A\}|$ and suppose that $m_0 < k$. 
We may assume that $A < A_1 \leq \dots \leq A_{m-m_0}$ and $A_i=A$ for $m-m_0+1 \leq i \leq m$.
Since the number $X$ is maximum among $\EE'$ such that $\pi(\EE') = \sigma$, 
to hold Lemma \ref{remaining} for $t_1, \dots, t_k  \gg 0$, 
namely for $\mm_1$ and $(\phi_1, \ep|_{\phi_1})$ appearing in the Langlands data of 
$\soc(S' \rtimes \pi) = \soc(\soc(S' \times \tau) \rtimes \sigma)$,
we need to have $m>m_0$ and $m' > 0$.
Moreover, by the same argument as in \cite[Theorem 5.2]{At-const},
we see that $([A,B_{m'}]_\rho, l'_{m'}, \eta'_{m'})$ and $([A_1,B]_\rho, l_1, \eta_1)$
satisfy one of conditions in \cite[Theorem 5.2]{At-const}.
Hence by this theorem, we can change $\EE$ so that 
$\{([A,B_{m'}]_\rho, l'_{m'}, \eta'_{m'}), ([A_1,B]_\rho, l_1, \eta_1)\}$ is replaced with 
\[
\{([A_1,B_{m'}]_\rho, l', \eta'), ([A,B]_\rho, l, \eta)\}
\]
for some $(l,\eta)$ and $(l',\eta')$.
Hence we can replace $m_0$ with $m_0+1$. 
Repeating this argument, we can achieve $\EE_0$ as in the assertion. 
\end{proof}

This completes the proof of Proposition \ref{typeA}.
Now, Theorem \ref{inductive} follows from Lemma \ref{pi'} and Proposition \ref{typeA}.

%\subsection{The initial step}
\subsection{The initial step}
Let $\pi \in \Irr^G$ be of good parity. 
In this section, we will prove that
if 
\[
D_{\rho|\cdot|^x}(\pi) \not= 0
\implies x \in \left\{0,\half{1}\right\} 
\]
for any $\rho^\vee \cong \rho$,
then $\pi$ is of Arthur type.
Indeed, we will prove a slightly more general statement. 
To state it, we need to introduce some notation.
\par

Let $\pi \in \Irr^G$ be of good parity. 
Write 
\[
\pi = L(\Delta_{\rho_1}[x_1,-y_1], \dots, \Delta_{\rho_r}[x_r,-y_r]; \pi(\phi, \ep))
\]
as in the Langlands classification, 
where 
\[
\phi = \bigoplus_{j=1}^s \rho'_j \boxtimes S_{2z_j+1}
\]
is a tempered $L$-parameter such that $\rho_1, \dots, \rho_r, \rho_1', \dots, \rho_s' \in \Cusp^\GL$ are all self-dual. 
Fix $\rho \in \Cusp^\GL$ with $\rho \cong \rho^\vee$. 
We set 
\begin{align*}
L_{\rho|\cdot|^x}(\pi) = L_x(\pi) &= \{i \in \{1,\dots,r\} \;|\; \rho_i \cong \rho,\; x_i = x\}, \\
R_{\rho|\cdot|^x}(\pi) = R_x(\pi) &= \{i \in \{1,\dots,r\} \;|\; \rho_i \cong \rho,\; y_i = x\}, \\
T_{\rho|\cdot|^x}(\pi) = T_x(\pi) &= \{j \in \{1,\dots,s\} \;|\; \rho'_j \cong \rho,\; z_j = x\}
\end{align*}
and 
\[
A_{\rho|\cdot|^x}(\pi) = A_x(\pi) = L_x \sqcup R_x \sqcup T_x.
\]
We will prove the following.

\begin{thm}\label{initial}
Notations are as above. 
For each self-dual $\rho \in \Cusp^\GL$, 
denote $b_\rho \in \{0,\half{1}\}$ such that 
$\rho \boxtimes S_{2b_\rho+1}$ is self-dual of the same type as $\phi$.
Assume that there are half-integers $a_\rho, a_1,\dots,a_k,b_1,\dots,b_k$ 
(depending on $\rho$) such that 
\begin{itemize}
\item
they are all congruent to $b_\rho$ modulo $\Z$; 
\item
we have 
\[
b_\rho \leq a_\rho < b_k \leq a_k < b_{k-1} \leq a_{k-1} < \dots < b_1 \leq a_1;
\]
\item
if $D_{\rho|\cdot|^x}(\pi) = 0$, then $x \in \{b_\rho, b_k,\dots,b_1\}$; 
\item
if $A_x(\pi) \not= \emptyset$, 
then $b_\rho \leq x \leq a_\rho$ 
or $b_i \leq x \leq a_i$ for some $1 \leq i \leq k$; 
\item
if $b_i \leq x \leq a_i$ for some $1 \leq i \leq k$, 
then $|A_x(\pi)| = 1$, so that $A_x(\pi)$ is equal to exactly one of 
$L_x(\pi)$, $T_x(\pi)$ or $R_x(\pi)$; 
\item
for each $1 \leq i \leq k$,  
\[
|\{x \;|\; b_i \leq x \leq a_i,\; A_x(\pi) = L_x(\pi)\}|
= 
|\{x \;|\; b_i \leq x \leq a_i,\; A_x(\pi) = R_x(\pi)\}|.
\]
\end{itemize}
Then $\pi$ is of Arthur type.
\end{thm}

When the implication $D_{\rho|\cdot|^x}(\pi) \not= 0 \implies x \in \{0,\half{1}\}$ holds 
for any $\rho \cong \rho^\vee$, 
then $\pi$ satisfies the conditions of Theorem \ref{initial} with $k=0$, 
and hence $\pi$ is of Arthur type.
\par

We prove Theorem \ref{initial} by induction on 
\[
X(\pi) = \sum_\rho \sum_{b_\rho < x \leq a_\rho} |A_{\rho|\cdot|^x}(\pi)|.
\]
Suppose first that this number is zero.
Then there are $\eta_\rho, \eta_i \in \{\pm1\}$ (which depend on $\rho$)
such that 
if we set 
\[
\EE = \bigcup_\rho \left(
\{
\underbrace{([b_\rho,b_\rho]_\rho,0,\eta_\rho), \dots, ([b_\rho,b_\rho]_\rho,0,\eta_\rho)}_{|A_{b_\rho}(\pi)|}
\}
\cup 
\bigcup_{i=1}^k \{([a_i,b_i]_\rho,l_i,\eta_i)\}
\right)
\]
with 
\[
l_i = 
|\{x \;|\; b_i \leq x \leq a_i,\; A_x(\pi) = L_x(\pi)\}|
= 
|\{x \;|\; b_i \leq x \leq a_i,\; A_x(\pi) = R_x(\pi)\}|, 
\]
then we have $\pi(\EE) = \pi$.
See Section \ref{sec.Moe}.
Therefore, $\pi$ is of Arthur type. 
\par

Now, we suppose that $A_{a_\rho}(\pi) \not= \emptyset$ 
for some $\rho \cong \rho^\vee$ and $a_\rho > b_\rho$. 
We may assume that $b_k$ is big enough so that 
there is an integer $t > 0$ such that 
\[
a_\rho+1 < b_\rho+t \leq a_\rho+t < b_k-1.
\]
Here, when $k = 0$, we ignore the inequality $a_\rho+t < b_k-1$.
Consider
\[
\pi' = \soc\left(
\begin{pmatrix}
b_\rho+t & \ldots & a_\rho+t \\
\vdots & \ddots & \vdots \\
b_\rho+1 & \ldots & a_\rho+1
\end{pmatrix}_\rho
\rtimes \pi
\right)
\]
which is irreducible by \cite[Proposition 3.4]{At-socle}.
By Tadi\'c's formula (Corollary \ref{TadicFormula}), we see that
\[
D_{\rho|\cdot|^x}(\pi') \not= 0 
\implies x \in \{b_\rho, b_\rho+t, b_k,\dots,b_1, -(a_\rho+1)\}.
\]

\begin{lem}\label{-(a+1)}
In the setting above, we have
\[
D_{\rho|\cdot|^{-(a_\rho+1)}}(\pi') = 0.
\]
\end{lem}
\begin{proof}
Suppose that $D_{\rho|\cdot|^{-(a_\rho+1)}}(\pi') \not= 0$.
Then the inclusion
\begin{align*}
\pi' \hookrightarrow 
\begin{pmatrix}
b_\rho+t & \ldots & a_\rho+t \\
\vdots & \ddots & \vdots \\
b_\rho+2 & \ldots & a_\rho+2
\end{pmatrix}_\rho
\times Z_\rho[b_\rho+1,a_\rho]
\times \rho|\cdot|^{a_\rho+1} \rtimes \pi
\end{align*}
factors through
\begin{align*}
\pi' \hookrightarrow 
\begin{pmatrix}
b_\rho+t & \ldots & a_\rho+t \\
\vdots & \ddots & \vdots \\
b_\rho+2 & \ldots & a_\rho+2
\end{pmatrix}_\rho
\times Z_\rho[b_\rho+1,a_\rho]
\rtimes \soc(\rho|\cdot|^{-(a_\rho+1)} \rtimes \pi).
\end{align*}
Hence
\begin{align*}
\pi = 
D_{\rho|\cdot|^{-(a_\rho+1)}}
\circ
D_{\rho|\cdot|^{b_\rho+1}, \dots, \rho|\cdot|^{a_\rho}}
\circ
\left(D_{\rho|\cdot|^{b_\rho+2}, \dots, \rho|\cdot|^{a_\rho+2}}
\circ \dots \circ 
D_{\rho|\cdot|^{b_\rho+t}, \dots, \rho|\cdot|^{a_\rho+t}}\right)(\pi').
\end{align*}
On the other hand, 
by definition of $\pi'$, 
we have
\begin{align*}
\pi = 
D_{\rho|\cdot|^{a_\rho+1}}
\circ
D_{\rho|\cdot|^{b_\rho+1}, \dots, \rho|\cdot|^{a_\rho}}
\circ
\left(D_{\rho|\cdot|^{b_\rho+2}, \dots, \rho|\cdot|^{a_\rho+2}}
\circ \dots \circ 
D_{\rho|\cdot|^{b_\rho+t}, \dots, \rho|\cdot|^{a_\rho+t}}\right)(\pi').
\end{align*}
These equations imply that 
\[
\soc(\rho|\cdot|^{-(a_\rho+1)} \rtimes \pi) \cong \soc(\rho|\cdot|^{a_\rho+1} \rtimes \pi), 
\]
which is equivalent to saying that $\rho|\cdot|^{a_\rho+1} \rtimes \pi$ is irreducible.
But since $A_{a_\rho}(\pi) \not= \emptyset$ and $A_{a_\rho+1}(\pi) = \emptyset$, 
this is impossible by \cite[Corollary 7.2]{AM}.
\end{proof}

Now we consider $A_x(\pi')$. 
Note that 
\[
|A_x(\pi')| = \left\{
\begin{aligned}
&|A_x(\pi)| -1 \iif b_\rho \leq x \leq a_\rho, \\
&1 \iif b_\rho+t \leq x \leq a_\rho+t, \\
&1 \iif b_i \leq x \leq a_i,\; 1 \leq i \leq k, \\
&0 \other.
\end{aligned}
\right. 
\]
In particular, for $b_\rho+t \leq x \leq a_\rho+t$, 
exactly one of 
\[
A_x(\pi') = L_x(\pi'), 
\quad 
A_x(\pi') = T_x(\pi'), 
\quad\text{or}\quad
A_x(\pi') = R_x(\pi')
\]
holds.
Moreover, since $D_{\rho|\cdot|^x}(\pi') = 0$ for $b_\rho+t < x \leq a_\rho+t$, 
by \cite[Theorem 7.1]{AM}, we see that 
\begin{itemize}
\item
if $x \geq b_\rho+t+1$ and $A_x(\pi') = L_x(\pi')$, then $A_{x-1}(\pi') = L_{x-1}(\pi')$; 
\item
if $x \leq a_\rho+t-1$ and $A_x(\pi') = R_x(\pi')$, then $A_{x+1}(\pi') = R_{x+1}(\pi')$.
\end{itemize}
Write
\begin{align*}
l_+ &= |\{b_\rho+t \leq x \leq a_\rho+t \;|\; A_x(\pi') = L_x(\pi')\}|, \\
l_- &= |\{b_\rho+t \leq x \leq a_\rho+t \;|\; A_x(\pi') = R_x(\pi')\}|.
\end{align*}
A key lemma is as follow.

\begin{lem}\label{l=l}
We have $l_+ = l_-$.
\end{lem}
\begin{proof}
First of all, we show that $l_+ \leq l_-$.
Indeed, for $x \geq b_\rho+t$, if $A_x(\pi') = L_x(\pi')$, then 
$\Delta_\rho[x,-y]$ appears in the Langlands data of $\pi'$ for some $y > x$. 
Hence $A_y(\pi') = R_y(\pi')$. 
The map $x \mapsto y$ induces an injection 
\[
\{x \geq b_\rho+t \;|\; A_x(\pi') = L_x(\pi')\}
\hookrightarrow 
\{y \geq b_\rho+t \;|\; A_y(\pi') = R_y(\pi')\}. 
\]
Since
\[
|\{x \;|\; b_i \leq x \leq a_i,\; A_x(\pi) = L_x(\pi)\}|
= 
|\{y \;|\; b_i \leq y \leq a_i,\; A_x(\pi) = R_x(\pi)\}|
\]
for $1 \leq i \leq k$ by assumption, 
we have $l_+ \leq l_-$.
In particular, if $l_- = 0$, then $l_+ = 0$ as well and we get $l_+ = l_-$.
\par

In the rest of the proof of this lemma, 
we assume that $l_- > 0$.
In particular, $A_{a_\rho+t}(\pi') = R_{a_\rho+t}(\pi')$. 
It implies by \cite[Corollary 7.2]{AM} that $L_{a_\rho}(\pi) = T_{a_\rho}(\pi) = \emptyset$.
Since $D_{\rho|\cdot|^x}(\pi) = 0$ for any $b_\rho < x \leq a_\rho$, 
we see that $\Delta_\rho[b_\rho,-a_\rho]$ appears in the Langlands data of $\pi$
with multiplicity exactly $k = |R_{a_\rho}(\pi)| = |A_{a_\rho}(\pi)|$. 
Namely, there is $\tau \in \Irr^G$ such that $\pi \hookrightarrow \Delta_\rho[b_\rho,-a_\rho]^k \rtimes \tau$
and such that $A_{a_\rho}(\tau) = \emptyset$.
\par

Now, for the sake of a contradiction, 
we assume that 
$\pi' \hookrightarrow \Delta_\rho[b_\rho,-a_\rho]^k \rtimes \tau'$ for some $\tau' \in \Irr^G$.
Then the inclusion 
\[
\pi' \hookrightarrow \begin{pmatrix}
b_\rho+t & \ldots & a_\rho+t \\
\vdots & \ddots & \vdots \\
b_\rho+1 & \ldots & a_\rho+1
\end{pmatrix}_\rho
\times \Delta_\rho[b_\rho,-a_\rho]^k \rtimes \tau
\]
must factor through
\begin{align*}
\pi' &\hookrightarrow 
\soc\left( \Delta_\rho[b_\rho,-a_\rho]^k \times
\begin{pmatrix}
b_\rho+t & \ldots & a_\rho+t \\
\vdots & \ddots & \vdots \\
b_\rho+1 & \ldots & a_\rho+1
\end{pmatrix}_\rho
\right) \rtimes \tau
\\
&\hookrightarrow 
\Delta_\rho[b_\rho,-a_\rho]^k \times
\begin{pmatrix}
b_\rho+t & \ldots & a_\rho+t \\
\vdots & \ddots & \vdots \\
b_\rho+1 & \ldots & a_\rho+1
\end{pmatrix}_\rho \rtimes \tau
\\&\hookrightarrow
\Delta_\rho[b_\rho,-a_\rho]^k \times
\begin{pmatrix}
b_\rho+t & \ldots & a_\rho+t \\
\vdots & \ddots & \vdots \\
b_\rho+2 & \ldots & a_\rho+2
\end{pmatrix}_\rho
\times Z_\rho[b_\rho+1,a_\rho] \times \rho|\cdot|^{a_\rho+1} \rtimes \tau.
\end{align*}
Since $A_{a_\rho}(\tau) = \emptyset$, by \cite[Corollary 7.2]{AM}, 
we see that $\rho|\cdot|^{a_\rho+1} \rtimes \tau$ is irreducible so that 
$\rho|\cdot|^{a_\rho+1} \rtimes \tau \cong \rho|\cdot|^{-(a_\rho+1)} \rtimes \tau$.
Then the above inclusion implies that $D_{\rho|\cdot|^{-(a_\rho+1)}}(\pi') \not= 0$. 
This contradicts Lemma \ref{-(a+1)}.
\par

By the last paragraph, 
one can see that the Langlands data of $\pi'$
contains $\Delta_\rho[x',-(a_\rho+t)]$ for some $b_\rho < x' < a_\rho+t$.
Then we see that $D_{\rho|\cdot|^{x'}}(\pi') \not= 0$. 
Hence we conclude that $x' = b_\rho+t$, i.e., 
the singleton $A_{b_\rho+t}(\pi')$ is equal to $L_{b_\rho+t}(\pi')$.
\par

If $A_y(\pi') = R_y(\pi')$ for some $b_\rho+t < y < a_\rho+t$, 
then the Langlands data of $\pi'$
contains $\Delta_\rho[x,-y]$ for some $b_\rho \leq x < y$.
We claim that $x > b_\rho+t$.
For the sake of a contradiction, 
we assume that there is $\Delta_\rho[x,-y]$ in the Langlands data of $\pi'$
such that $x \leq b_\rho+t$ and $b_\rho+t < y < a_\rho+t$.
We may assume that $y$ is maximum among this condition.
Then $\Delta_\rho[x',-(y+1)]$ is in the Langlands data of $\pi'$ for some $x' \geq b_\rho+t$.
Since the segment $[x,-y]_\rho$ is contained in $[x',-(y+1)]_\rho$, 
by \cite[Theorem 7.1]{AM}, we have $D_{\rho|\cdot|^{y+1}}(\pi') \not= 0$, 
which is a contradiction.
Therefore, for $b_\rho+t \leq y \leq a_\rho+t$ with $A_y(\pi') = R_y(\pi')$, 
if $\Delta_\rho[x,-y]$ is in the Langlands data of $\pi'$, 
then $b_\rho+t \leq x \leq a_\rho+t$. 
This gives the bijection 
\[
\{b_\rho+t \leq y \leq a_\rho+t \;|\; A_y(\pi') = R_y(\pi')\}
\rightarrow
\{b_\rho+t \leq x \leq a_\rho+t \;|\; A_x(\pi') = L_x(\pi')\}
\]
and we conclude that $l_+ = l_-$, as desired.
\end{proof}

Now we set $b_{k+1} = b_\rho+t$ and $a_{k+1} = a_\rho+t$.
Then $\pi'$ satisfies all of the assumptions in Theorem \ref{initial} 
for $a_\rho, a_1, \dots, a_{k+1}, b_1, \dots, b_{k+1}$. 
Moreover, since 
\[
\sum_\rho \sum_{b_\rho < x \leq a_\rho} |A_{\rho|\cdot|^x}(\pi)|
-
\sum_\rho \sum_{b_\rho < x \leq a_\rho} |A_{\rho|\cdot|^x}(\pi')|
= a_\rho-b_\rho > 0, 
\]
we can apply the inductive hypothesis to $\pi'$. 
Hence we can write $\pi' = \pi(\EE')$. 
Moreover, we may assume that 
\[
\EE' \ni ([a_\rho+t,b_\rho+t], l, \eta)
\]
for some $\eta \in \{\pm1\}$, where $l=l_+=l_-$ is defined in Lemma \ref{l=l}. 
Then by construction, 
with 
\[
\EE = (\EE' \setminus \{([a_\rho+t,b_\rho+t], l, \eta)\}) \cup \{([a_\rho,b_\rho], l, \eta)\}, 
\]
we have $\pi = \pi(\EE)$.
Therefore, we conclude that $\pi$ is of Arthur type.
This completes the proof of Theorem \ref{initial}.
\par

Now Theorem \ref{unitary_gp} follows from Theorems \ref{inductive} and \ref{initial}.

%\section{Slightly beyond case}
%\section{Slightly beyond case}
\section{Slightly beyond case}
In this section, we treat a case slightly beyond the good parity case, 
which may be important for global applications.
\par

%\subsection{Statement}
\subsection{Statement}
We know that if $\psi \in \Psi(G)$, then all members in $\Pi_\psi$ are unitary. 
However, due to the lack of a proof of the Ramanujan conjecture, 
if $\Psi$ is a global $A$-parameter of a classical group $\mathbb{G}$ over a number field $\F$, 
and if $v$ is a place of $\F$ such that $\F_v = F$ and $G = \mathbb{G}(\F_v)$, 
then the localization $\psi = \Psi_v$ belongs to $\Psi^+(G)$ but may not lie in $\Psi(G)$. 
\par

By the weak Ramanujan bound, which is know, 
it takes the form 
\[
\psi = \psi_0 \oplus 
\bigoplus_{i=1}^r (\rho_i|\cdot|^{x_i} \oplus \rho_i^\vee|\cdot|^{-x_i}) \boxtimes S_{c_i} \boxtimes S_{d_i},
\]
where 
\begin{itemize}
\item
$\psi_0 = \psi_{\mathrm{good}}$; 
\item
$\rho_i \in \Cusp^\GL$ is unitary; 
\item
$0 \leq x_i < \tfrac{1}{2}$; 
\item
if $x_i = 0$, then $\rho_i \boxtimes S_{c_i} \boxtimes S_{d_i}$ is not self-dual of the same type as $\psi$.
\end{itemize}
Here, $r$ may be zero, in which case $\psi = \psi_{\mathrm{good}}$ is of good parity. 
For such a parameter $\psi$, 
its $A$-packet $\Pi_\psi$ is defined as the set of irreducible parabolic inductions
\[
\pi = \bigtimes_{i=1}^r \Sp(\rho_i, c_i,d_i)|\cdot|^{x_i} \rtimes \pi_0
\]
for $\pi_0 \in \Pi_{\psi_0}$.
\par

In this section, we prove the following theorem, which was conjectured in \cite[Conjecture 5.9]{HJLLZ}. 

\begin{thm}\label{weak}
Let 
\[
\pi = \bigtimes_{i=1}^r \Sp(\rho_i, c_i,d_i)|\cdot|^{x_i} \rtimes \pi_0
\]
where $\pi_0$ is of Arthur type of good parity, 
$\rho_i$ is unitary, and $0 \leq x_i < \tfrac{1}{2}$ for $1 \leq i \leq r$. 
Suppose that $\pi$ is irreducible. 
For each $\rho \in \Cusp^\GL$ and each pair of positive integers $(c,d)$, 
we set 
\[
I(\rho,c,d) = \{i \in \{1,\dots,r\} \;|\; \rho_i \cong \rho,\; (c_i,d_i) = (c,d)\}.
\]
Then $\pi$ is unitary if and only if 
for each $\rho$ and $(c,d)$, the following conditions hold:
\begin{itemize}
\item
If $\rho \not\cong \rho^\vee$, then 
\[
\{x_i \;|\; i \in I(\rho,c,d),\; x_i\not= 0\} = \{x_i \;|\; i \in I(\rho^\vee,c,d),\; x_i\not= 0\}
\]
as multi-sets.
\item
If $\rho \cong \rho^\vee$ and $|I(\rho,c,d)|$ is odd, then 
the unitary induction $\Sp(\rho, c,d) \rtimes \pi_0$ is irreducible. 
\end{itemize}
\end{thm}

For the proof, we use (UI), (UR), (CS) and (RP1) in Proposition \ref{list} frequently. 

%\subsection{Proof of unitarity}
\subsection{Proof of unitarity}
Here, we prove the ``if'' part of Theorem \ref{weak}.
Namely we show that if $\pi$ satisfies the conditions in Theorem \ref{weak}, 
then $\pi$ is unitary. 
The proof is by induction on $r$. 
\par

First, suppose that $x_i = 0$ for some $i$.
Since $\pi$ is irreducible, by \cite[Theorem 1.2]{BS}, we see that $\Sp(\rho_i,c_i,d_i) \rtimes \pi_0$ is also irreducible. 
Then since $\pi' = \times_{j \not= i} \Sp(\rho_j, c_j,d_j)|\cdot|^{x_j} \rtimes \pi_0$ is unitary by induction hypothesis, 
so is $\pi = \Sp(\rho_i,c_i,d_i) \rtimes \pi'$ by (UI).
In the rest of this subsection, we assume that $x_i > 0$ for any $i$.
\par

Fix $\rho$ unitary and a pair $(c,d)$. 
Suppose that $r \geq 2$ and that there are $i$ and $j$ with $i \not=j$ 
such that $\rho_i \cong \rho$, $\rho_j \cong \rho^\vee$ and $(c_i,d_i) = (c_j,d_j) = (c,d)$.
We may assume that $(i,j)=(1,2)$, $x_1 \geq x_2$ 
and that $x_1=x_2$ if $\rho \not\cong \rho^\vee$. 
By induction hypothesis, 
\[
\pi' = \bigtimes_{i=3}^r \Sp(\rho_i, c_i,d_i)|\cdot|^{x_i} \rtimes \pi_0
\]
is unitary. 
Since $0 < x_2 < \half{1}$, by \cite{T-0}, 
$\Sp(\rho,c,d)|\cdot|^{x_2} \times \Sp(\rho,c,d)|\cdot|^{-x_2}$ is unitary. 
Hence 
\[
\Sp(\rho,c,d)|\cdot|^{x_2} \times \Sp(\rho,c,d)|\cdot|^{-x_2} \rtimes \pi'
\cong 
\Sp(\rho,c,d)|\cdot|^{x_2} \times \Sp(\rho^\vee,c,d)|\cdot|^{x_2} \rtimes \pi'
\]
is an irreducible unitary representation by (UI).
In particular, if $\rho \not\cong \rho^\vee$, then $\pi$ is unitary. 
If $\rho \cong \rho^\vee$ and if we set 
\[
\Pi_t = \Sp(\rho,c,d)|\cdot|^{t} \times \Sp(\rho,c,d)|\cdot|^{x_2} \rtimes \pi', 
\]
then for $x_2 \leq t \leq x_1$, 
we see that $\Pi_t$ is irreducible hermitian representation since $\overline\rho \cong \rho^\vee \cong \rho$. 
Since $\Pi_{x_2}$ is unitary, by (CS), we see that $\pi = \Pi_{x_1}$ is also unitary. 
\par

Hence we may assume that $\rho_i \cong \rho_i^\vee$ for any $1 \leq i \leq r$, 
and that $I(\rho,c,d)$ is a singleton or empty for each $\rho \cong \rho^\vee$ and $(c,d)$. 
For $(t_1, \dots, t_r)$ with $0 \leq t_i \leq x_i$, we consider
\[
\Pi_{(t_1,\dots,t_r)} = \bigtimes_{i=1}^r \Sp(\rho_i, c_i,d_i)|\cdot|^{t_i} \rtimes \pi_0. 
\]
By assumption together with \cite[Theorem 1.2]{BS}, 
we see that $\Pi_{(0,\dots,0)}$ is irreducible. 
In particular, $\Pi_{(t_1,\dots,t_r)}$ is an irreducible hermitian representation 
for any $(t_1, \dots, t_r)$ with $0 \leq t_i \leq x_i$. 
Since $\Pi_{(0, \dots, 0)}$ is unitary, by (CS), 
we see that $\pi = \Pi_{(x_1, \dots,x_r)}$ is also unitary. 
This completes the proof of the ``if'' part of Theorem \ref{weak}. 

%\subsection{Proof of non-unitarity}
\subsection{Proof of non-unitarity}
In this subsection, we will prove the ``only if'' part by induction on $r$. 
Suppose that $\pi$ is unitary. 
Our goal is to show the conditions in Theorem \ref{weak}.
\par

First, we assume that $x_i = 0$ for some $i$. 
Put $\pi' = \times_{j \not= i} \Sp(\rho_j, c_j,d_j)|\cdot|^{x_j} \rtimes \pi_0$.
Then $\Sp(\rho_i,c_i,d_i) \boxtimes \pi'$ is an irreducible hermitian representation 
of a Levi subgroup of $G$ such that its induction $\Sp(\rho_i,c_i,d_i) \rtimes \pi'$ is irreducible and unitary. 
Hence by (UR), we see that $\Sp(\rho_i,c_i,d_i) \boxtimes \pi'$ is also unitary. 
In particular, $\pi'$ is unitary. 
Repeating this argument, we may assume that $x_i > 0$ for $1 \leq i \leq r$.
\par

Next, we assume that $\rho \not\cong \rho^\vee$.
Since $\pi$ is hermitian, we see that 
\begin{align*}
\overline{\pi} &\cong \bigtimes_{i=1}^r \Sp(\overline{\rho_i}, c_i,d_i)|\cdot|^{x_i} \rtimes \overline{\pi_0}
\cong \bigtimes_{i=1}^r \Sp(\rho_i^\vee, c_i,d_i)|\cdot|^{x_i} \rtimes \pi_0^\vee
\\&\cong  \left(\bigtimes_{i=1}^r \Sp(\rho_i, c_i,d_i)|\cdot|^{-x_i} \rtimes \pi_0\right)^\vee
\cong  \left(\bigtimes_{i=1}^r \Sp(\rho_i^\vee, c_i,d_i)|\cdot|^{x_i} \rtimes \pi_0\right)^\vee
\end{align*}
is isomorphic to $\pi^\vee$.
Hence $|I(\rho,c,d)| = |I(\rho^\vee, c,d)|$ and 
$\{x_i \;|\; i \in I(\rho,c,d)\} = \{x_i \;|\; i \in I(\rho^\vee,c,d)\}$ as multi-sets.
Moreover, if $\rho_1 \cong \rho$, $\rho_2 \cong \rho^\vee$, 
$(c_1,d_1) = (c_2,d_2) = (c,d)$ and $x_1=x_2$, 
then since 
\[
\pi \cong \Sp(\rho,d,c)|\cdot|^{x_2} \times \Sp(\rho,d,c)|\cdot|^{-x_2} \rtimes \pi'
\]
is an irreducible unitary representation with 
$\pi' = \times_{i=3}^r \Sp(\rho_i, c_i,d_i)|\cdot|^{x_i} \rtimes \pi_0$, 
by (UR), 
we see that 
\[
\left(\Sp(\rho,c,d)|\cdot|^{x_2} \times \Sp(\rho,c,d)|\cdot|^{-x_2}\right) \boxtimes \pi'
\]
is also unitary. 
In particular, $\pi'$ is unitary. 
Repeating this argument, we may assume that $\rho_i \cong \rho_i^\vee$ for $1 \leq i \leq r$.
\par

Suppose that $\rho \cong \rho^\vee$. 
If $\rho_1 \cong \rho_2 \cong \rho$ and $(c_1,d_1) = (c_2,d_2) = (c,d)$, 
then by (CS), we see that 
\[
\Sp(\rho,c,d)|\cdot|^{x_2} \times \Sp(\rho,c,d)|\cdot|^{x_2} \rtimes \pi'
\cong 
\Sp(\rho,c,d)|\cdot|^{x_2} \times \Sp(\rho,c,d)|\cdot|^{-x_2} \rtimes \pi'
\]
is an irreducible unitary representation with 
$\pi' = \times_{i=3}^r \Sp(\rho_i, c_i,d_i)|\cdot|^{x_i} \rtimes \pi_0$.
By (UR), we see that $\pi'$ is unitary. 
Repeating this argument, 
we may assume that $I(\rho,c,d)$ is at most a singleton for each $\rho \cong \rho^\vee$ and $(c,d)$.
\par

If $\Sp(\rho_1,c_1,d_1) \rtimes \pi_0$ is irreducible, 
then by \cite[Theorem 1.2]{BS}, we see that $\Sp(\rho_1,c_1,d_1) \rtimes \pi'$ is also irreducible 
with $\pi' = \times_{i=2}^r \Sp(\rho_i, c_i,d_i)|\cdot|^{x_i} \rtimes \pi_0$.
By (CS) and (UR), the unitarity of $\pi$ implies that $\pi'$ is also unitary. 
Repeating this argument, 
we may assume that $\Sp(\rho_i,c_i,d_i) \rtimes \pi_0$ is reducible for $1 \leq i \leq r$.
\par

We have proven that we may now assume the following:
\begin{enumerate}
\item
$\pi = \times_{i=1}^r \Sp(\rho_i, c_i,d_i)|\cdot|^{x_i} \rtimes \pi_0$ is an irreducible unitary representation; 
\item
$\rho_i \cong \rho_i^\vee$ and $0 < x_i < \half{1}$ for $1 \leq i \leq r$; 
\item
for $i \not= j$, we have $\rho_i \not\cong \rho_j$ or $(c_i,d_i) \not= (c_j,d_j)$; 
\item
$\Sp(\rho_i, c_i,d_i) \rtimes \pi_0$ is reducible for $1 \leq i \leq r$.
\end{enumerate}
The goal is to show that $r = 0$ in this case. 
For the sake of a contradiction, we assume that $r>0$.

\begin{prop}
Continue in the above setting. 
Let $P = MN$ be the maximal parabolic subgroup of $G$
such that $\pi_M = (\times_{i=1}^r \Sp(\rho_i, c_i,d_i)) \boxtimes \pi_0 \in \Irr(M)$. 
Then there is an $A$-parameter $\psi_M \in \Psi(M)$ 
with $\pi_M \in \Pi_{\psi_M}$ such that 
$R_P(w, \pi_M, \psi_M)$ is not a scalar operator, 
where $w$ is the unique non-trivial Weyl element in $W(M,G)$.
\end{prop}
\begin{proof}
By \cite[Corollary 3.33]{BS}, 
there exist $1 \leq i \leq r$ and an irreducible summand $\pi'$ of 
\[
\bigtimes_{\substack{1 \leq j \leq r \\ j \not= i}} \Sp(\rho_j, c_j,d_j) \rtimes \pi_0
\]
such that $\Sp(\rho_i, c_i,d_i)  \rtimes  \pi'$ is reducible. 
For simplicity, set $\rho = \rho_i$ and $(c,d) = (c_i,d_i)$. 
Since $\pi'$ is unitary and of good parity, we can write $\pi' = \pi(\EE')$ 
for some extended multi-segment $\EE'$. 
In \cite[Definition 3.4]{BS}, Bo{\v s}njak--Stadler introduced a modification of $\EE'$ given by 
\[
\FF^{-1} \colon
\EE' = \bigcup_{\rho'}\{([A_i,B_i]_{\rho'},l_i,\eta_i)\}_{i \in I_{\rho'}} \mapsto 
\Sc' = \bigcup_{\rho'}\{([A_i,B_i]_{\rho'},\mu_i)\}_{i \in I_{\rho'}}
\]
with 
\[
\mu_i = \left(\prod_{j=1}^{i-1}(-1)^{A_j-B_j}\right) \eta_i (A_i-B_i+1-2l_i).
\]
We may assume that $\Sc'$ is standard in the sense of \cite[Definition 3.15]{BS}.
Then by \cite[Corollary 3.30]{BS}, 
there are two irreducible summands $\pi_1$ and $\pi_2$ of $\Sp(\rho, c,d) \rtimes \pi'$
associated to $\Sc'_1 = \Sc \cup \{([C,D]_\rho,\nu), ([C,D]_\rho,\nu)\}$ and 
$\Sc'_2 = \Sc \cup \{([C,D]_\rho,\nu+2), ([C,D]_\rho,\nu+2)\}$, respectively, 
for some $\nu \in \Z$ with $\nu \equiv d \bmod 2$, 
where $C = \half{c+d}-1$ and $D = \half{c-d}$.
Note that $\Sc'_1$ and $\Sc'_2$ give the same $A$-parameter $\psi \in \Psi(G)$ 
as in \cite[Definition 3.3(d)]{BS}, which satisfies that $\pi_1,\pi_2 \in \Pi_\psi$. 
Moreover, we may assume that 
$\psi$ comes from an $A$-parameter $\psi_M \in \Psi(M)$ so that $\pi_M \in \Pi_{\psi_M}$.
That is, 
\[
\psi = \left(\bigoplus_{j=1}^r \rho_j \boxtimes S_{c_j} \boxtimes S_{d_j}\right)^{\oplus2} \oplus \psi_0
\]
with $\pi_0 \in \Pi_{\psi_0}$. 
By the description of $\FF$ in \cite[Definition 3.4]{BS} together with \cite[Theorem 3.6]{At-const} 
(corrected in \cite[Appendix A]{At-Is_Arthur}), 
we see that
\[
\prod_{j=1}^r \frac{\pair{e(\rho_j,c_j,d_j), \pi_1}_\psi}{\pair{e(\rho_j,c_j,d_j), \pi_2}_\psi}
= (-1)^{\half{|\nu+2|-|\nu|}} \left((-1)^{d-1}\frac{\sgn(\nu)}{\sgn(\nu+2)}\right)^{d}
\]
with the convention $\sgn(0) = 1$.
A case-by-case computation shows that the right-hand side is always equal to $-1$.
\par

Now by the local intertwining relation proven by Arthur \cite[Theorem 2.4.1]{Ar}, 
the normalized local intertwining operator $R_P(w,\pi_M,\psi_M)$ acts on $\pi_k$ by 
the scalar $\prod_{j=1}^r \pair{e(\rho_j,c_j,d_j), \pi_k}_\psi$ for $k=1,2$.
Therefore, the eigenvalue of $R_P(w,\pi_M,\psi_M)$ on $\pi_1$ differs from the one on $\pi_2$.
\end{proof}

In particular, (RP1) in Proposition \ref{list} implies that for sufficiently small $s > 0$, 
the irreducible representation $\times_{i=1}^r \Sp(\rho_i, c_i,d_i)|\cdot|^{s} \rtimes \pi_0$ is not unitary. 
However, (CS) together with the unitarity of $\pi$ would imply that this representation should be unitary.
This is a contradiction.
This completes the proof of the ``only if'' part of Theorem \ref{weak}. 

%\section{Remarks on the further beyond}
%\section{Remarks on the further beyond}
\section{Remarks on the further beyond}\label{sec.beyond}
We have described the unitary dual of split $\SO_{2n+1}(F)$ and $\Sp_{2n}(F)$ 
in the good parity case and a slightly beyond. 
In the further beyond case, several new phenomena arise, which we now explain in this section.

%\subsection{One parameter complementary series}
\subsection{One-parameter complementary series}
The first step beyond the good parity case is to consider the set of representations of the form 
\[
\Pi_s = \Sp(\rho,a,b)|\cdot|^s \rtimes \pi_0,
\]
where 
\begin{itemize}
\item $\pi_0$ is of Arthur type and of good parity; 
\item $\rho \in \Cusp^\GL$ with $\rho^\vee \cong \rho$; 
\item $s > 0$.
\end{itemize}
To study this, it is important to consider 
the \emph{first reducibility point}  
\[
\FRP(\Sp(\rho,a,b), \pi_0) = \inf\left\{s \geq 0 \;\middle|\; 
\text{$\Sp(\rho,a,b)|\cdot|^s \rtimes \pi_0$ is reducible}\right\}.
\]
It is a half-integer and can be computed algorithmically (\cite[Theorem 5.3, Corollary 5.4]{At-socle}).
If $s \equiv \FRP(\Sp(\rho,a,b), \pi_0) \bmod \Z$, 
then all irreducible subquotients of $\Pi_s$ are of good parity. 
We need to consider the other case. 

\begin{qu}
Let $\pi_0 \in \Irr^G$ be of good parity and of Arthur type. 
For $s > 0$ with $s \not\equiv \FRP(\Sp(\rho,a,b), \pi_0) \bmod \Z$, 
if the irreducible representation $\Sp(\rho, a,b)|\cdot|^s \rtimes \pi_0$ is unitary, 
then must it be that $s < \FRP(\Sp(\rho,a,b), \pi_0)$?
\end{qu}

The converse follows from (CS).
Note that if $s$ is larger than $\FRP(\Sp(\rho,a,b), \pi_0)$ but close to it, 
the non-unitarity would follow if one could establish an analytic property of intertwining operators 
(see \cite[(RP), Section 2]{MT}).
\par

To deal with cases where $s$ is much larger than $\FRP(\Sp(\rho,a,b), \pi_0)$, 
we need more explicit control over $\FRP(\Sp(\rho,a,b), \pi_0)$ than is currently available in \cite{At-socle}. 
The difficulty is that $\FRP(\Sp(\rho,a,b), \pi_0)$ really depends on the pair $(a,b)$, 
unlike in the earlier settings treated in \cite{LMT} and \cite{MT}.

%\subsection{Some examples}
\subsection{Some examples}
Finally, we give some explicit examples of unitary representations, beyond the good parity case.
In this subsection, we set $\rho = \1_{\GL_1(F)}$, and we drop $\rho$ from the notations.
\par

Let $\sigma$ be the unique irreducible supercuspidal representation 
in the $A$-packet $\Pi_\psi$ associated to $\psi = S_5 \boxtimes S_5 \in \Psi(\Sp_{24}(F))$.
Then $\FRP(\Sp(2,2), \sigma) = 4$.
Fix $0 < \epsilon < \half{1}$.
Since the induced representation 
\[
\Sp(2,2)|\cdot|^{-s} \rtimes \sigma
\]
is irreducible for $0 \leq s \leq 2+\epsilon$, 
we see that $\Sp(2,2)|\cdot|^{2+\epsilon} \rtimes \sigma$ is unitary. 
We consider 
\[
\Pi_s = \Sp(2,2)|\cdot|^{-(2+\epsilon)} \times \Sp(2,2)|\cdot|^{-s} \rtimes \sigma.
\]
It is irreducible for $0 \leq s < \epsilon$.
Therefore, all irreducible subquotients of $\Pi_\epsilon$ are unitary. 
Note that $\Sp(2,2)|\cdot|^{-(2+\epsilon)} \times \Sp(2,2)|\cdot|^{-\epsilon}$ 
contains $L(\mm)|\cdot|^{-(1+\epsilon)}$ as an irreducible subquotient, 
where $\mm$ is Leclerc's example (see Remark \ref{Max}) given by 
\[
\mm = [-1,-2]+ [0,0]+ [1,-1]+ [2,1].
\]
Hence
\[
L(\mm)|\cdot|^{1+\epsilon} \rtimes \sigma = 
L(\Delta[-2-\epsilon,-3-\epsilon], |\cdot|^{-1-\epsilon}, \Delta[-\epsilon,-2-\epsilon], 
\Delta[\epsilon,\epsilon-1]; \sigma)
\]
is unitary. 
However, it is difficult to see from the Langlands data alone why the right-hand side is unitary. 
This suggests that further notions may be necessary to fully describe the unitary dual of classical groups.
\par

Additional difficulties arise when considering the limit $\epsilon \to \frac{1}{2}$. 
Since all irreducible subquotients of the complementary series are unitary, 
we conclude that
\[
\pi = L\left(\Delta\left[-\tfrac{5}{2},-\tfrac{7}{2}\right], |\cdot|^{-\frac{3}{2}}, 
\Delta\left[-\tfrac{1}{2},-\tfrac{5}{2}\right], \Delta\left[\tfrac{1}{2},-\tfrac{1}{2}\right]; \sigma\right)
\]
is also unitary. 
\par

We see from the above argument that $\pi$ is not isolated in the unitary dual. 
However, we do not know how to prove this fact directly from the Langlands data, 
since the induced representation
\[
\Pi = L\left(\Delta\left[-\tfrac{5}{2},-\tfrac{7}{2}\right], |\cdot|^{-\frac{3}{2}}, 
\Delta\left[-\tfrac{1}{2},-\tfrac{5}{2}\right], \Delta\left[\tfrac{1}{2},-\tfrac{1}{2}\right]\right) \rtimes \sigma
\]
is reducible. 
(To see that, one can check that $\Pi$ is preserved under Aubert duality, whereas $\pi$ is not.)

%References

\end{document}